\documentclass[12pt]{article}
\usepackage[english]{babel}
\usepackage{amsfonts,amsmath,amssymb,amsthm,graphicx,afterpage,url}
\input{xy}
\xyoption{all}

\usepackage{hyperref}
\hypersetup{
   colorlinks   = true, 
   urlcolor     = blue, 
   linkcolor    = blue, 
   citecolor   = red 
}

\newcommand{\jonly}[1]{}
\newcommand{\aronly}[1]{#1}

\def\rk{\mathop{\fam0 rk}}

\def\t{\widetilde}
\def\R{{\mathbb R}} \def\Z{{\mathbb Z}}

\theoremstyle{theorem}
\newtheorem{theorem}{Theorem}[section]
    \newtheorem{lemma}[theorem]{Lemma}
    
    \newtheorem{proposition}[theorem]{Proposition}

\theoremstyle{definition}
\newtheorem{remark}[theorem]{Remark}

\begin{document}

\title{A short proof 
of the Patak-Tancer theorem on non-embeddability of $k$-complexes in $2k$-manifolds\footnote{We are grateful for useful discussions to S. Dzhenzher, R. Karasev, M. Tancer and S. Zhilina.}}
 
\author{E. Kogan and A. Skopenkov
\footnote{Supported by Russian Science Foundation, grant N 25-21-00685. 
\newline
E. Kogan: Higher School of Economics.
\newline
A. Skopenkov: Moscow Institute of Physics and Technology, Independent University of Moscow.
Email: \texttt{skopenko@mccme.ru}.
\texttt{https://users.mccme.ru/skopenko/}.}}

\date{}

\maketitle

\begin{abstract}
In 2019 P. Patak and M. Tancer obtained the following higher-dimensional generalization of the Heawood inequality on embeddings of graphs into  surfaces.
We present a short well-structured proof accessible to non-specialists in the field.

Let $\Delta_n^k$ be the union of $k$-dimensional faces of the $n$-dimensional simplex.

{\bf Theorem.} {\it
(a) If $\Delta_n^k$ PL embeds into the connected sum of $g$ copies of the Cartesian product $S^k\times S^k$ of two $k$-dimensional spheres, then $g\ge\dfrac{n-2k-1}{k+2}$.

(b) If $\Delta_n^k$ PL embeds into a closed $(k-1)$-connected PL $2k$-manifold $M$, then
$(-1)^k(\chi(M)-2)\ge\dfrac{n-2k-1}{k+1}$.}
\end{abstract}


\aronly{\tableofcontents}

\section{Introduction}

The classical Heawood inequality states that {\it if the complete graph $K_n$ on $n$ vertices is embeddable (i.e. realizable without self-intersections) in the sphere with $g$ handles, then}
$$g\ge\frac{(n-3)(n-4)}{12}.$$
Denote by $\Delta_n^k$ the union of $k$-dimensional faces of the $n$-dimensional simplex (this is the body of the complete $(k+1)$-regular hypergraph on $n+1$ vertices).
A higher-dimensional analogue of the Heawood inequality is the K\"uhnel conjecture whose simplified version states that 
{\it if $\Delta_n^k$ embeds into the connected sum of $g$ copies of the Cartesian product $S^k\times S^k$ of two $k$-dimensional spheres, then}
$$g\ge\frac{(n-k-2)(n-k-3)\ldots(n-2k-2)}{2(k+1)(k+2)\ldots(2k+1)}.$$
We present a short well-structured proof
(see Remark \ref{r:earlier}) of the Pat\'ak-Tancer \cite{PT19} estimate
$$g\ge\frac{n-2k-1}{k+2}.$$
For this estimate one needs the following definitions and a more general result.

The {\it Euler characteristics} of a PL $2k$-manifold $M$ is $\chi(M)=a_0-a_1+a_2-\ldots+a_{2k}$, where $a_j$ is the number of $j$-dimensional faces in some (or, equivalently, in any) triangulation (or cell decomposition) of $M$.
A PL manifold $M$ is called {\it $(k-1)$-connected} if for any $j=0,1,\ldots,k-1$ any continuous map $\Delta_{j+1}^j\to M$ extends to a continuous map $\Delta_{j+1}^{j+1}\to M$.

\begin{theorem}[{\cite[Theorem 1]{PT19}}]\label{t:patan}
If $\Delta_n^k$ PL embeds into a closed $(k-1)$-connected PL $2k$-manifold $M$, then
$(-1)^k(\chi(M)-2)\ge \dfrac{n-2k-1}{k+1}$.

If, moreover, the intersection form of $M$ is even, then
\linebreak
$(-1)^k(\chi(M)-2)\ge \dfrac{2(n-2k-1)}{k+2}$.
\end{theorem}

For a definition of the \textit{homology group} $H_k(M;\Z_2)$ and the \textit{intersection form}
$\cap_M : H_k(M;\Z_2)\times H_k(M;\Z_2) \to \Z_2$  accessible to non-specialists in topology (in particular, to computer scientists) 
see \cite[\S1.1, \S5.3]{Pr07},   \cite[\S6, \S10]{Sk20}, \cite{IF, HG}.
A bilinear form $q:V\times V\to \Z_2$ on a $\Z_2$-vector space 
$V$  is called \textbf{even} if $q(v,v)$ is even for every $v\in V$. 

The above estimate $g\ge\frac{n-2k-1}{k+2}$ is a particular case of Theorem \ref{t:patan} because
the intersection form of the connected sum $M$ of $g$ copies of $S^k\times S^k$ is even, and $(-1)^k(\chi(M)-2)=2g$.

By Remark \ref{r:patan}.c our proof also recovers \cite[Theorem 2 and Corollary 3]{PT19}  (but not \cite[Theorems 4, 6 and 10]{PT19}, see Remark \ref{r:earlier}).
Our exposition was useful to obtain stronger results \cite{DS22, SS23} on (and related to) the K\"uhnel conjecture (of Remark \ref{r:patan}.a).

\begin{remark}\label{r:patan}
(a) The K\"uhnel conjecture asserts that {\it if $\Delta_n^k$ PL embeds into a closed $(k-1)$-connected PL $2k$-manifold $M$, then}
$$(-1)^k(\chi(M)-2)\ge \frac{(n-k-2)(n-k-3)\ldots(n-2k-2)}{(k+1)(k+2)\ldots(2k+1)}.$$
Giving a survey on this K\"uhnel conjecture is beyond the scope of the present paper; see e.g. the surveys \cite[\S1, \S6]{DS22}, \cite{Ku23}.  

(b) For a closed $(k-1)$-connected PL $2k$-manifold $M$ we have $(-1)^k(\chi(M)-2)=\rk H_k(M;\Z_2)$.
If in Theorem \ref{t:patan} we replace $(-1)^k(\chi(M)-2)$ by $\rk H_k(M;\Z_2)$, then the $(k-1)$-connectedness assumption could be omitted. 
(This holds by our proof, and holds for \cite[Theorem 1]{PT19}.)

(c) A map $f:\Delta_n^k\to Y$ to a subset $Y\subset\R^d$ is called an {\it almost embedding} 
if $f\sigma\cap f\tau=\emptyset$ for any non-adjacent faces  $\sigma,\tau$.
A general position PL map $f:\Delta_n^k\to M$ to a $2k$-manifold $M$ is called a {\it $\Z_2$-embedding} if
$|f\sigma\cap f\tau|$ is even for any non-adjacent faces $\sigma,\tau$ (see well-known definition of `general position' e.g. in \cite[\S1]{Sk24}).
Theorem \ref{t:patan} holds (with the same proof) if one replaces `embeds into' by `has an almost embedding to' or even by `has a $\Z_2$-embedding to'. 
This non-$\Z_2$-embeddability result is stronger than Theorem \ref{t:patan} by \cite{FK17} and \cite[Theorem 1.2.1.b]{Sk24}.
For $k=1$ this result asserts that {\it if $K_s$ is $\Z_2$-embeddable to a surface $M$, then $3\chi(M)\le2-2s$}.
This is covered (at least for large $s$) by \cite[Theorem 1]{FK19}, because $K_{n,n}$ is a subgraph of $K_{2n}$.

(d) In the following paragraph we prove that \emph{the intersection form of any \emph{smooth} $(k-1)$-connected $2k$-manifold is even when $k\ne1,2,4$}. 
The (presumably folklore) proof is written by D. Crowley and A. Skopenkov.
We conjecture that the same result holds for {\it PL} manifolds.

Assume that $k\ge3$ and $k\ne4$.
Let $M$ be given manifold.
Take any $x\in H_k(M;\Z_2)$. Since $M$ is $(k-1)$-connected and $k\ge3$, we can realize $x$ as a smooth embedding $S^k\to M$.
Let $\overline e$ be the modulo 2 Euler number of the normal bundle of this embedding.
Then $x^2=\overline e=0\in\Z_2$.
Here the first equality is an exercise and the second equality is a version of the celebrated result on non-parallelizability of spheres \cite{BM58}. 

\end{remark}


\begin{remark}\label{r:earlier}
We make the proof of Theorem \ref{t:patan} well-structured by splitting it into independent parts: 

$\bullet$ the purely topological part (Theorem \ref{l:odd}, cf. Remark \ref{r:low} and \cite[Proposition 16.C2]{PT19} referred later in \cite{PT19} as Proposition 16.(ii)), and

$\bullet$ the purely algebraic part (Theorem \ref{p:rank}, cf. \cite[Proposition 17]{PT19}). 

This is done by introducing the notion of an ${U\choose l}$-matrix. 
Our proof of Theorem \ref{t:patan} is shorter and hopefully clearer than \cite{PT19} because of this, and because 

$\bullet$ we present elementary statement and proof of Theorem \ref{p:rank} 
(without using cycles and homology as in \cite[Proposition 17 and \S4.2]{PT19}, and topological terminology as in \cite[Lemma 20]{PT19}); 

$\bullet$ we highlight the main idea of Theorem \ref{p:rank} by explicitly stating purely algebraic Lemmas \ref{l:step-alga} and \ref{l:step-alg}
(which we prove in a slightly simpler way, not just in the same way using simpler terminology); 


$\bullet$ we do not use `obstruction machinery' from \cite[\S1.2, \S3]{PT19}, as opposed to \cite[\S4.1]{PT19}
(for a simplified exposition of this machinery see \cite[\S1]{Sk24}, \cite[\S2.5]{KS21e}).

Our exposition clarifies the relation of the proof from \cite{PT19} to earlier known results (see footnote \ref{f:rela} and \aronly{Remark \ref{r:rankher}}\jonly{\cite[Remark 3.7]{KS21}}), 
and to the low-rank matrix completion problem \cite{NKS}.
\end{remark}


\begin{remark}\label{r:low} Here by a low-dimensional example we motivate the notion of a ${U\choose l}$-matrix, and illustrate the idea of 
Theorem \ref{l:odd}.
This remark is not formally used in the proof.

Consider the following statements:

(\ref{r:low}.A) {\it Any four pairwise distinct points 
on the circle can be split into two `intertwined' pairs in a unique way.}

(\ref{r:low}.B) {\it For any pairwise distinct points $A_1,A_2,A_3,A_4$ on the circle}
$$|A_1A_2\cap A_3A_4|+|A_1A_3\cap A_2A_4|+|A_1A_4\cap A_2A_3|=1.$$

(\ref{r:low}.B') {\it For any general position PL map $f:K_5\to\R^2$ the number of intersection points in $\R^2$ formed by images of disjoint edges is odd.}

(\ref{r:low}.C) {\it Take any embedding 
$f\colon K_5\to M$ to a 2-surface $M$.
Take any map $f'\colon K_n\to M$ in general position to $f$.
For any pairwise distinct numbers $i,j,k\in[n]$ take a cycle $\langle ijk\rangle$ in $K_n$.
Denote $ijk\wedge pqr := |f\langle ijk\rangle\cap f'\langle pqr\rangle|$. 
Then} 
$$123\wedge234 + 123\wedge235 + 123\wedge245 +  123\wedge345 \underset2\equiv 0 \quad and$$
$$125\wedge345 + 135\wedge245 + 145\wedge235 \underset2\equiv 1.$$

Here $(A)$ and $(A)\Rightarrow(B)\Rightarrow(C)$ are easy.
A simple proof of $(A)\Rightarrow(B')$ is presented in \cite{Sk14}
(for the linear case, for the PL case the deduction is analogous).
\end{remark}


We shorten $\{i\}$ to $i$ and `a symmetric square matrix with $\Z_2$-entries' to just `matrix'. 
For $l\ge3$ and a finite set $U$, a \textbf{${U\choose l}$-matrix} is a matrix 
whose rows and whose columns correspond to all $l$-element subsets of $U$, and for which the following properties hold:

\textbf{(independence)} $A_{P,Q}=0$ if $P\cap Q=\emptyset$;

\textbf{(linear dependence)} for each $(l+1)$-element and $l$-element subsets $F,P\subset U$
$$\sum\limits_{i\in F} A_{F-i,P}=0;$$

\textbf{(non-triviality)} for each $i\in U$ and $(2l-2)$-element subset $F\subset U-i$ we have $A_{F,i}=1$, where
$$A_{F,i} := \sum\limits_{\{X,Y\}\ :\ F\cup i=X\cup Y,\ X\cap Y=i,\ |X|=|Y|=l} A_{X,Y} =
\sum\limits_{\{\sigma,\tau\}\ :\ F=\sigma\sqcup \tau,\ |\sigma|=|\tau|=l-1} A_{i\sqcup \sigma,i\sqcup\tau}.$$
Analogously to Remark (\ref{r:low}.C) (by Theorem \ref{l:odd} below), an ${[n+1]\choose k+2}$-matrix is constructed by a $\Z_2$-embedding of $\Delta^k_n$ to a $2k$-dimensional manifold.



Let  $[m]:=\{1,2,\ldots,m\}$.   
We identify $(s+1)$-element subsets of $[n+1]$ with $s$-dimensional faces of $\Delta_n$.
For a $(k+2)$-element subset $P\subset [n+1]$ denote by $\partial P$ the boundary sphere of the $(k+1)$-dimensional face
$P$.
For a $2k$-manifold $M$, a map $f\colon \Delta_n^k\to M$, and $(k+2)$-element subsets $P,Q\subset [n+1]$
denote
$$A(f)_{P,Q} := f\partial P\cap_M f\partial Q\in\Z_2.$$
(Here $\cap_M$ the algebraic intersection modulo 2 of $k$-dimensional spheres, or of modulo 2 homology $k$-cycles, see the sentence after Theorem \ref{t:patan}.) 
The obtained square matrix $A(f)$ of size $n+1\choose k+2$ is symmetric.

\begin{theorem}[{\cite[Proposition 16]{PT19}}; proved in \S\ref{s:prodd}]\label{l:odd}
If $f:\Delta_n^k\to M$ is an embedding into a $2k$-manifold $M$, then $A(f)$ is a symmetric $[n+1]\choose k+2$-matrix.
\end{theorem}



\begin{theorem}[cf. {\cite[Proposition 17]{PT19}}; proved in \S\ref{s:prank}]\label{p:rank}
If $l\ge3$ and $A$ is an ${[m]\choose l}$-matrix, then 
$\rk A\ge\dfrac{m-2l+2}{l-1}$.

If, moreover, $A$ is even, then $\rk A\ge\dfrac{2(m-2l+2)}l$.
\end{theorem}

A 
matrix 
is called {\it even} if all the numbers on the main diagonal are even.

\begin{proof}[Proof of Theorem \ref{t:patan} assuming Theorems \ref{p:rank} and \ref{l:odd}]
For a PL embedding $f:\Delta_n^k\to M$ the matrix $A=A(f)$ is the Gramian matrix (with respect to $\cap_M$)
of the homology classes of the images $f\partial P$.
Hence $\rk H_k(M;\Z_2)\ge\rk A$ by the following well-known result\aronly{ (see a proof e.g. in \cite[Lemma 2.1]{Bi21})}.

{\it Let $v_1, v_2, \ldots, v_s$ be vectors in some $d$-dimensional linear space over $\Z_2$ with a bilinear symmetric product.
Let $A$ be the Gramian matrix of $v_1, v_2, \ldots, v_s$. 
Then $d\ge\rk A$.}

By Theorems \ref{l:odd} and \ref{p:rank} we obtain $\rk A\ge\frac{n-2k-1}{k+1}$ and,
for $A$ even, $\rk A\ge\frac{2(n-2k-1)}{k+2}$.
So we are done by Remark \ref{r:patan}.b.
\end{proof}

\section{Proof of Theorem \ref{l:odd}}\label{s:prodd}

The symmetry, the independence and the linear dependence are straightforward.
Just as (\ref{r:low}.C) is easily implied by (\ref{r:low}.A), the non-triviality is reduced to a higher-dimensional analogue of (\ref{r:low}.A), i.e. to Theorem \ref{t:stat-ilpl}.(odd) below\footnote{\label{f:rela} In \cite[\S4.1]{PT19} Theorem \ref{l:odd} was essentially reduced to Theorem \ref{t:stat-ilpl}.(even), which is a higher-dimensional analogue of (\ref{r:low}.B').
For a direct proof of the equivalence of the odd and the even cases of Theorem \ref{t:stat-ilpl} see
\cite[Remark 4.1.aeg]{Sk16}
\newline
Theorem \ref{l:odd} is a `conical version' of  Theorem \ref{t:stat-ilpl}.(odd).
The `conical version' of  the analogue of Theorem \ref{t:stat-ilpl}.(odd) for the join of $(d+1)/2$ copies of the four-point set (in place of the $(d+1)/2$-skeleton of $\Delta_{d+2}$) is \cite[Lemma 1']{Sk03}.
This analogue coincides with Theorem \ref{t:stat-ilpl}.(odd) for $d=1$ but is different for $d>1$.
This analogue is closer to (\ref{r:low}.B) than to (\ref{r:low}.C).}.
Namely, it suffices to prove the non-triviality for $n=2k+2$.
For a $(k+1)$-element subset $\sigma\subset[2k+2]$ let $\widehat\sigma:=\partial(\{2k+3\}\sqcup\sigma)$;
this is a $k$-dimensional sphere in $\Delta_{2k+2}^k$.
Then the non-triviality for $n=2k+2$ follows because
$$\sum f\widehat\sigma\cap_M f\widehat\tau \ \overset{(1)}=\
\sum |(D\cap f\widehat\sigma)\cap(D\cap f'\widehat\tau)|\mod2\ \overset{(2)}=\
\sum L_{\sigma, \tau} \ \overset{(3)}=\ 1.$$
Here

$\bullet$ the sums are over $\{\{\sigma,\tau\}\ :\ [2k+2]=\sigma\sqcup \tau,\ |\sigma|=k+1\}$;

$\bullet$  $f'\colon \Delta_{2k + 2}^k\to M$ is an embedding close to $f$ and in general position to $f$;

$\bullet$ $D \subset M$ is a small $2k$-disk in general position to $f\Delta_{2k + 2}^k$ and to
$f'\Delta_{2k + 2}^k$, containing $f(2k + 3)$ and $f'(2k + 3)$,
and such that $D \cap f\Delta_{2k + 2}^k = D \cap f'\Delta_{2k + 2}^k = \emptyset$;

$\bullet$ the equality $(1)$ holds because $f$ is an embedding and by definition of $f'$;

$\bullet$ $L_{\sigma, \tau}$ is the modulo 2 linking number of $\partial D\cap f\widehat\sigma$
and $\partial D \cap f'\widehat\tau$; by the symmetry of the linking number
$L_{\sigma, \tau} = L_{\tau, \sigma}$, hence the sum over $\{\sigma, \tau\}$ is well-defined;

$\bullet$ the equality $(2)$ holds by the following well-known statement (cf. \cite[Lemma 2]{Sk18o}):
{\it for any proper general position PL maps $g,g':D^k\to B^{2k}$ such that
$g\partial D^k\cap g'\partial D^k=\emptyset$ the number $|gD^k\cap g'D^k|$ modulo 2 equals to the modulo 2
linking number of $g\partial D^k$ and $g'\partial D^k$ in $\partial B^{2k}$.}

$\bullet$ the equality $(3)$ is Theorem \ref{t:stat-ilpl}.(odd).

\begin{theorem}
\label{t:stat-ilpl}
Let $f:\Delta_{d+2}\to\R^d$ be a general position PL map.

(odd) If $d$ is odd, then the number of unordered pairs of disjoint $(d+1)/2$-faces the images of whose boundaries are linked modulo 2, is odd.

(even) If $d$ is even, then the number of intersection points in $\R^d$ formed by images of disjoint $(d/2)$-faces is odd.
\end{theorem}

Theorem \ref{t:stat-ilpl}

$\bullet$ is (\ref{r:low}.A) for $d=1$,

$\bullet$ is the well-known Conway-Gordon-Sachs theorem for $d=3$, i.e. for graph $K_6$ in 3-space,

$\bullet$ is its higher-dimensional version \cite{SS92}\footnote{Theorem \ref{t:stat-ilpl}.(odd) is not explicitly stated in \cite{SS92}
but can be proved analogously to \cite[Lemma 1.4]{SS92}. See details in \cite[Theorem 1.7 and footnote 4]{KS20}.},
\cite{LS98, Ta00} for $d>3$ odd, and

$\bullet$ is the van Kampen-Flores theorem for $d$ even, see e.g. survey \cite[\S4]{Sk16}.

This is because any general position PL map $\Delta_{d+2}^{\lfloor\frac d2\rfloor}\to\R^d$ extends to a general position PL map  $\Delta_{d+2}\to\R^d$, and because for $d$ odd any general position PL map $\Delta_{d+2}^{\frac{d-1}2} \to \R^d$ is an embedding.

\section{Proof of Theorem \ref{p:rank}}\label{s:prank}

Theorem \ref{p:rank} is deduced by induction from Propositions \ref{p:rankest}.a,b below,
which is deduced from the following lemmas explaining the key idea of proof.

\begin{lemma}[trivial]\label{p:hered} Let $B$ be the square matrix of size ${m-1\choose l}$ obtained from an ${[m]\choose l}$-matrix by deleting rows and columns corresponding to all subsets containing $m$.
Then $B$ is an ${[m-1]\choose l}$-matrix.
\end{lemma}


\begin{lemma}\label{l:step-alga} Let $A$ be a 
matrix 
whose rows and whose columns correspond to all $l$-element subsets of $[m]$, and for which the independence property holds. 
Assume that $A_{X,X}=A_{Y,Y}=0$ and $A_{X,Y}=1$ for some $X,Y\subset[m]$. 
Let $A|_{[m]-X}$ be the `\emph{restriction}' of $A$ to $[m]-X$, i.e. the matrix of size ${m-l\choose l}$ 
obtained from $A$ by deleting rows and columns corresponding to subsets containing at least element from $X$.
Then $\rk A\ge\rk A|_{[m]-X}+2$.

\end{lemma}

\begin{proof}[Proof (suggested by S. Zhilina)] 
Let $A'$ be the `restriction' of $A$ to $X, Y$ and all $l$-element subsets of $[m]-X$. 
Then the $X$-line and the $X$-row of $A'$ consist of zeros except for $A'_{X,Y}=A'_{Y,X}=1$. 
Subtracting the $X$-line from some lines corresponding to subsets of $[m]-X$,  
we obtain a matrix $A''$  whose $Y$-row consist of zeros except for $A''_{X,Y}=1$. 
Then subtracting the $X$-row from some rows corresponding to subsets of $[m]-X$,  
we obtain a matrix $A'''$  whose $Y$-line also consist of zeros except for $A'''_{Y,X}=1$. 
The `restrictions' of $A'''$ and of $A$ to all $l$-element subsets of $[m]-X$ coincide.  
Since $A_{X,Y}=1$, by the independence $X\cap Y\ne\emptyset$, so $A|_{[m]-X}$ does not contain the $Y$-line (and the $Y$-row). 
Then $\rk A \ge \rk A' = \rk A''' = \rk A|_{[m]-X} + 2$. 
\end{proof}

\begin{lemma}[proved below; cf. \aronly{Remark \ref{r:step-algb}}\jonly{\cite[Remark 3.6]{KS21}}]\label{l:step-alg} 
Let $A$ be an ${[m]\choose l}$-matrix. 
Assume that $A_{X,X}=1$ for some $X\subset[m]$. 
Let $U$ be the union of $[m]-X$ and any element from $X$. 
For $l$-element subsets $P,Q\subset U$ define 
$$C_{P,Q}:=A_{P,Q}+A_{P,X}A_{Q,X}.$$
Then $\rk C<\rk A$ and $C$ is a $U\choose l$-matrix.
\end{lemma}


Denote by $r_m$ the minimal rank of an ${[m]\choose l}$-matrix.
Denote by $\t{r_m}$ the minimal rank of an even ${[m]\choose l}$-matrix.
Clearly, $r_m=\t{r_m}=0$ for $m\le2l-2$, both sequences $r_m,\t{r_m}$ are non-decreasing, and $r_m\le\t{r_m}$.

\begin{proposition}\label{p:rankest} (a) $\t{r_m}\ge\t{r_{m-l}}+2$;

(b) $r_m\ge\min\{r_{m-l+1}+1,\t{r_m}\}$ (more precisely, either $r_m=\t{r_m}$ or $r_m\ge r_{m-l+1}+1$).
\end{proposition}

\begin{proof}[Proof of (a)]
Take an even $[m]\choose l$ matrix $A$ such that $\rk A=\t{r_m}$.
By the non-triviality $A\ne0$.
Then there are $l$-element subsets $X,Y\subset[m]$ for which the assumptions of Lemma \ref{l:step-alga} are fulfilled.
Then
$$\t{r_m} = \rk A \overset{(1)}\ge \rk A|_{[m]-X}+2 \overset{(2)} \ge \t{r_{m-l}}+2,\quad\text{where}$$


$\bullet$ the inequality (1) follows from Lemma \ref{l:step-alga};

$\bullet$ the inequality (2) holds because $A|_{[m]-X}$ is even, and is an $[m]-X\choose l$-matrix by Lemma \ref{p:hered}.
\end{proof}

\begin{proof}[Proof of (b) modulo Lemma \ref{l:step-alg}]
Take an ${[m]\choose l}$-matrix $A$ such that $\rk A = r_m$.
If $A$ is even, then $r_m = \t{r_m}$, so we are done.
Otherwise there is a $l$-element subset $X\subset [m]$ such that $A_{X,X}=1$.
Then by Lemma \ref{l:step-alg}
$$r_m = \rk A \overset{(1)}\ge \rk C + 1\overset{(2)}\ge r_{m - l + 1} + 1, \quad\text{where}$$

\aronly{$\bullet$ $C$ is the matrix defined in Lemma \ref{l:step-alg};}

$\bullet$ the inequality (1) follows from Lemma \ref{l:step-alg};

$\bullet$ the inequality (2) holds because $C$ is an $[m]-X\choose l$-matrix by Lemma \ref{p:hered}.
\end{proof}

\begin{proof}[Proof of Theorem \ref{p:rank} modulo Lemma \ref{l:step-alg}]
Theorem \ref{p:rank} asserts that $r_m\ge\dfrac{m-2l+2}{l-1}$ and $\t{r_m}\ge\dfrac{2(m-2l+2)}l$.

The inequality for $\t{r_m}$ is proved by induction on $m$.
The base $m\le 2l - 2$ is clear.
By Proposition \ref{p:rankest}.a and the inductive hypothesis we have
$$\t{r_m} \geqslant \t{r_{m-l}}+2 \geqslant \frac{2(m - 3l + 2)}{l} + 2 = \dfrac{2(m-2l+2)}l.$$

The inequality for $r_m$  is now proved by induction on $m$.
The base $m\le 2l - 2$ is clear.
By the inequality for $\t{r_m}$, Proposition \ref{p:rankest}.b and the inductive hypothesis we have
$$r_m \geqslant \min\left\{r_{m-l+1}+1,\t{r_m}\right\} \geqslant
\min\left\{\frac{m-3l+3}{l-1}+1,\frac{2(m-2l+2)}l\right\} = \frac{m-2l+2}{l-1}.$$
\end{proof}
 
\begin{proof}[Proof of Lemma \ref{l:step-alg}]
\aronly{\footnote{Before reading this proof one can warm up by reading Remark \ref{r:step-algb}.a.}}
In this paragraph we prove that $\rk C<\rk A$ (as suggested by S. Zhilina).
Let $A'$ be the `restriction' of $A$ to $X$ and all $l$-element subsets of $U$. 
Subtracting the $X$-line from some lines corresponding to subsets of $U$,  
we obtain a matrix $A''$  whose $X$-row consist of zeros except for $A''_{X,X}=1$. 
Then subtracting the $X$-row from some rows corresponding to subsets of $U$,  
we obtain a matrix $A'''$  whose $X$-line also consist of zeros except for $A'''_{X,X}=1$. 
The restriction of $A'''$ to all $l$-element subsets of $U$ coincides with $C$. 
Thus $\rk C = \rk A'''-1 = \rk A'-1 \le \rk A-1$. 

In this paragraph we prove that $C$ satisfies the independence property.
If $P \cap Q = \emptyset$, then either $P \cap X = \emptyset$ or $Q \cap X = \emptyset$.
Hence $C_{P, Q} = A_{P, Q} + A_{P, X}A_{Q, X} = 0 + 0 = 0$.

In this paragraph we prove that $C$ satisfies the linear dependence property.
For each $(l+1)$-element and $l$-element subsets $F,P\subset U$ we have
$$\sum\limits_{i \in F} C_{F - i, P} =
\sum\limits_{i \in F} A_{F - i, P} + A_{P, X} \sum\limits_{i \in F} A_{F - i, X} = 0.$$

In this paragraph we prove that $C$ satisfies the non-triviality property.
By Lemma \ref{p:rankher} below we may assume that $i$ is not the element of $U-X$.
So for every summand $C_{i\sqcup \sigma,i\sqcup\tau}$ of $C_{F,i}$, at least one of the sets $i\sqcup \sigma,i\sqcup\tau$ 
does not contain the element of $U-X$, hence that set does not intersect $X$.
Then by the independence $C_{i\sqcup \sigma,i\sqcup\tau} = 
A_{i\sqcup \sigma,i\sqcup\tau} + A_{i\sqcup \sigma, X}A_{i\sqcup\tau, X} = 
A_{i\sqcup \sigma,i\sqcup\tau}$.
Thus $C_{F,i}=A_{F,i}=1$.
\end{proof}

Recall that $C_{F,i}$ (in the notation $A_{F,i}$) is defined in the non-triviality property.

\begin{lemma}[cf. \aronly{Remark \ref{r:rankher}}\jonly{{\cite[Remark 3.7]{KS21}}}]\label{p:rankher}
Let $C$ be a 
matrix 
whose rows and whose columns correspond to all $l$-element subsets of a finite set $U$, and for which the linear dependence property holds. 
Then for each $i\in U$ and $(2l-2)$-element subset $F\subset U-i$ the residue $C_{F,i}$ depends only on $F\sqcup i$ not on $(F,i)$.
\end{lemma}

\begin{proof}
It suffices to prove that $C_{G\sqcup i,j}=C_{G\sqcup j,i}$ for each $i,j\in U$ and $(2l-3)$-element subset $G\subset U-i-j$.
Denote $\overline\sigma := \{i, j\}\sqcup\sigma$.
Then
$$C_{G\sqcup j,i}+C_{G\sqcup i,j} \overset{(1)}=
\sum\limits_{\{(\sigma,\tau)\ :\ G=\sigma\sqcup \tau,\ |\sigma|=l-2\}}
\left(C_{\overline\sigma, i\sqcup\tau} + C_{\overline\sigma, j\sqcup\tau}\right) \overset{(2)}=$$
$$= \sum\limits_{\{(\sigma,\tau)\ :\ G=\sigma\sqcup \tau,\ |\sigma|=l-2\}}\ \sum\limits_{t \in \tau} C_{\overline\sigma, \overline{\tau - t}}\overset{(3)}=
\sum\limits_{t\in G}\ \sum\limits_{\{(\sigma,\nu)\ :\ G-t=\sigma\sqcup\nu,\ |\sigma|=l-2\}}
C_{\overline\sigma, \overline\nu}
\overset{(4)}=0,\quad\text{where}$$

$\bullet$ the equality (1) holds because $C_{G\sqcup j,i}$ is equal to the sum of the first summands $C_{\overline\sigma,i\sqcup\tau}$,
and $C_{G\sqcup i,j}$ is equal to the sum of the second summands $C_{\overline\sigma,j\sqcup\tau}$;

$\bullet$ the equality (2) holds by the linear dependence for $F=\overline\tau$, $P=\overline\sigma$;

$\bullet$ the equality (3) is obtained by changes of the order of summation, and of variable $\nu=\tau-t$;

$\bullet$ the equality (4) holds because ordered decompositions $(\sigma,\nu)$ of $G-t$ into $(l-2)$-element subsets $\sigma,\nu$ split into pairs $\{(\sigma,\nu),(\nu,\sigma)\}$ and $C_{\overline\sigma,\overline\nu}+C_{\overline\nu,\overline\sigma}=0$.
\end{proof}

\aronly{

\begin{remark}\label{r:step-algb} 
(a) Let us present a simple proof that $r_m\ge\min\{r_{m-l}+1,\t{r_m}\}$ (more precisely, either $r_m=\t{r_m}$ or $r_m\ge r_{m-l}+1$). 
(This estimate is slightly weaker than Proposition \ref{p:rankest}.b, gives an estimate slightly weaker than Theorem \ref{p:rank}, and illustrates 
Lemma \ref{l:step-alg}.) 


Take an $[m]\choose l$ matrix $A$ such that $\rk A=r_m$.
If $A$ is even, then $r_m=\t{r_m}$, so we are done.
Otherwise there is a $l$-element subset $X\subset [m]$ such that $A_{X,X}=1$.
Let $A'$ be the `restriction' of $A$ to $X$ and to $l$-element subsets of $[m]-X$.
Then
$$r_m=\rk A\ge\rk A'\overset{(1)}=\rk A|_{[m]-X}+1\overset{(2)}\ge r_{m-l}+1,\quad\text{where}$$

$\bullet$ the equality (1) holds because $A_{X,X}=1$ and by the independence $A'_{X,Z}=0$ for any $Z\subset [m]-X$;

$\bullet$ the inequality (2) holds because $A|_{[m]-X}$ is an $[m]-X\choose l$-matrix by Lemma \ref{p:hered}.

(b) Proposition \ref{p:rankest}.b is implied by the following conjecture (more natural than Lemma \ref{l:step-alg}). 

\emph{Let $A$ be a 
matrix 
whose rows and whose columns correspond to all $l$-element subsets of $[m]$, and for which the independence holds. 
Assume that $A_{X,X}=1$ for some $X\subset[m]$. 
Then there is an $(m-l+1)$-element subset $U\subset[m]$ such that $\rk A|_U<\rk A$ for the 
restriction $A|_U$ of $A$ to $U$. 
(Perhaps as $U$ one can take the union of $[m]-X$ and any element from $X$.)} 

However, a discussion with S. Zhilina shows that the conjecture is apparently false. 
\end{remark}

\begin{remark}\label{r:rankher}
Lemma \ref{p:rankher} is a combinatorial version of \cite[Lemma 20]{PT19}, and generalizes the following long known fact to higher dimensions, 
i.e. to $(k+1)$-element subsets of $[2k+3]$:


{\it Denote by $X={{[5]\choose2}\choose2}$ the set of unordered pairs of 2-element subsets of $[5]$.
For any $i\in[5]$ and a partition $[5]-i=\sigma\sqcup\tau$ into disjoint 2-element sets denote
$$T_{i,\{\sigma,\tau\}}:=\{\{\alpha,\beta\}\in X\ :\  \alpha\subset\sigma\sqcup i,\ \beta\subset\tau\sqcup i\}.$$
Denote by $C_i$ the sum modulo 2 of sets $T_{i,\{\sigma,\tau\}}$ over all non-ordered partitions $[5]-i=\sigma\sqcup\tau$ as above.
Then
$$C_i=\{\{\alpha,\beta\}\in X\  :\ \alpha\cap\beta=\emptyset\}$$
and so is independent of $i$.}


For this fact see e.g. surveys \cite[Lemmas 2.2'.c and 2.8'.c]{Sk14}, \cite[Lemma 2.3.6 for $K=K_5$]{MNS}. 
\aronly{If 2-element subsets of $[5]$ are considered as edges of $K_5$, then
$C_i$ is \emph{the quotient deleted product} of $K_5$.
In the notation of the non-triviality for $U=[5]$, $l=3$, $F=F_i=[5]-i$, a general position map $f:K_5\to\R^2$,
and $A_{P,Q}=f\partial P\cap f\partial Q$ we have that $C_{F_i,i}$ is the scalar product with $C_i$ of
\emph{the intersection cocycle} of $f$, so $C_{F_i,i}$ is \emph{the van Kampen number} of $f$.}
\end{remark}

\begin{proof}[An alternative proof of Lemma \ref{l:step-alga} (closer to \cite{PT19})]
Take a basis of $\Z_2^{{m\choose l}}$ corresponding to $l$-element subsets of $[m]$.
Define a bilinear form $A$ on $\Z_2^{{m\choose l}}$ by setting $A(P,Q):=A_{P,Q}$ for basic vectors $P,Q$.
Take any $l$-element set $P\subset[m]$.
Let
$$\overline P=\overline P(X,Y):=P+A_{X,P}Y+A_{Y,P}X.$$
Recall that
$$(*)\qquad A_{X,Y}=A_{Y,X}=1\quad\text{and}\quad A_{X,X}=A_{Y,Y}=0.$$
Hence
$$(**)\qquad A(\overline P,X)=A(\overline P,Y)=0$$
(i.e., $\overline P$ is the orthogonal projection of $P$ to the orthogonal complement of $\left<X,Y\right>$ with respect to $A$).
By the independence, for $P\subset[m]-X$ we have $\overline P=P+A_{Y,P}X$.
Hence for every $l$-element sets $P,Q\subset [m]-X$ we have
$$(***)\qquad A(\overline P,\overline Q) = A_{P, Q}+0+0+0 = B_{P, Q}.$$
(I.e., $B$ is the Gramian matrix with respect to $A$ of the `projections' $\overline P$ of $l$-element sets $P\subset[m]-X$.)

Let $B'$ be the Gramian matrix with respect to $A$ of $X,Y$ and the `projections' $\overline R$ of $l$-element sets $R\subset[m]-X$.
I.e., $B'_{P,Q}=A(\widehat P,\widehat Q)$, where $\widehat P=P$ if $P\in\{X,Y\}$, and $\widehat P=\overline P$ otherwise; $\widehat Q$ is defined analogously.
Then

$\bullet$ $B'_{X,Y} = B'_{Y,X} = 1$, $B'_{X, X} = B'_{Y, Y} = 0$ (by (*)),

$\bullet$ $B'_{X,P} = B'_{P,X} = B'_{Y,P} = B'_{P,Y} = 0$ for $P\ne X,Y$ (by (**)), and

$\bullet$ $B'_{P,Q} = B_{P,Q}$ for $P,Q\subset[m]-X$ (by (***)).

Hence $\rk B+2 = \rk B' \le \rk A$.
\end{proof}

\begin{proof}[An alternative proof that $\rk C<\rk A$ in Lemma \ref{l:step-alg} (closer to \cite{PT19})]
Take a basis of $\Z_2^{{m\choose l}}$ corresponding to $l$-element subsets of $[m]$.
Define a bilinear form $A$ on $\Z_2^{{m\choose l}}$ by setting $A(P,Q):=A_{P,Q}$ for basic vectors $P,Q$.
Let $P_X$ be the orthogonal projection of $P\in \Z_2^{{m\choose l}}$ to the orthogonal complement of $X$ (with respect to $A$), i.e. $P_X:=P+A_{P,X}X$.
We have
$$A(P_X,Q_X) = A(P,Q)+A(A_{P,X}X,Q)+A(P,A_{Q,X}X)+A(A_{P,X}X,A_{Q,X}X) =$$
$$= A_{P,Q}+A_{P,X}A_{X,Q}+A_{P,X}A_{Q,X}+A_{P,X}A_{Q,X}A_{X,X} = A_{P,Q} + A_{P,X}A_{Q,X} = C_{P,Q}.$$
Then $C$ is the Gramian matrix (with respect to $A$) of the projections of subsets of $U$.
Let $C'$ be the Gramian matrix (with respect to $A$) of $X$ and the projections of subsets of $U$.
We have $C_{P, Q} = C'_{P, Q}$ whenever $P, Q \subset U$.
Furthermore, $C'_{X,P} = C'_{P, X} = 0$ for any basic vector $P \ne X$, and $C'_{X, X} = A_{X, X} = 1$.
Thus $\rk C = \rk C' - 1 <\rk A$.
\end{proof}

}

\section{Appendix}

\begin{remark}\label{r:lett} This remark is formed by public letters discussing the questions

$\bullet$ if there is a conflict of interest in publication of the expository part of this paper;

$\bullet$ should we update an arxiv version of our paper when the current arxiv version is under review;

$\bullet$ should a paper in a refereed journal conceal or reveal its methodology.

(The first question was raised by M. Tancer, see letter of May 8 below, and the other two turned out to be relevant.)

Discussion of our actions in such practical situations reveals our understanding `for whom research is done' \cite{Sk21d}.
Such an understanding is sometimes concealed, so we are grateful to M. Tancer for stating his opinion, in spite of it being partly different from ours.

These letters might be interesting as an example of an open discussion of a controversial question,
carried in full mutual respect of participants of the discussion, and leaving the final decision to the reader.
(See the motivation for publicity in the letter of May 13, 2021 below.)

No reply to the letter of June 5
was received by the time of arxiv submission of this paper.
So the discussion seems to have its final form ready for a reader's judgement.
(If available, an update of this discussion will be presented here.)

\small

\smallskip
{\it (A. Skopenkov to P. Patak and M. Tancer,
May 8, 2021)}

Dear Martin,  Dear Pavel,

Martin raised the `conflict of interests' question in his February letter.
So in our paper with Eugene we need to publish our opinion on that,
and to publicly invite you to publish your opinion.
It would be nice if we could find a phrase that satisfies all of us.
However, there is nothing wrong to present our different opinions so that a reader could make his/her own judgement.
Below please find some suggestions for Remark 2f to be added to our paper.
Could you please either choose any of them, or suggest your own phrase we could agree with,
or make your own public statement?

{\it (1) We are grateful to M. Tancer and P. Patak for confirming that there is no conflict of interest in publication of this paper, in spite of A. Skopenkov has been a referee of \cite{PT3} since March to August of 2020.

(2) We are grateful to M. Tancer and P. Patak for confirming that there is no conflict of interest in publication of this paper, in spite of A. Skopenkov has been a referee of \cite{PT3} since March to August of 2020.
This is so because

$\bullet$ the paper \cite{PT3} is openly published on arxiv, and so is available for praise, for criticism, as well as for building upon, properly mentioning the authors' contribution;

$\bullet$  the current paper properly mentions the contribution of \cite{PT3}, stating that it only presents an exposition of the Partak-Tancer results;

$\bullet$ I am not a referee of \cite{PT3} since August of 2020;

$\bullet$  the current paper is based upon suggestions I made in frame of
my being referee of \cite{PT3}, and the authors have chosen not to realize those suggestions, at least in their full extent leading to a short exposition presented here. (Even the improved version of \cite{PT3} partly realizing these suggestions from summer of 2020 is not available to the math community in May of 2021 and the authors kindly informed me in Fall 2020 and again in May, 2021 that they have no estimation for the date when they will make that improved version available to the math community.\footnote{[added in 2022] The improved version of \cite{PT3} is now available as \cite{PT19}.})


(3) In our opinion, there is no conflict of interest in publication of this paper,  in spite of A. Skopenkov has been a referee for \cite{PT3} since March to August of 2020.
This is because [bullet points from (2)].
We are grateful to M. Tancer and P. Patak for learning our opinion presented above and stating that they do find a conflict of interest in publication of this paper because
[here a text from MT and PP is to be presented].

(4) In our opinion, there is no conflict of interest in publication of this paper, in spite of A. Skopenkov has been a referee for \cite{PT3} since March to August of 2020.
This is because [bullet points from (2)].
We asked M. Tancer and P. Patak to read our opinion presented above and publicly state if they find a conflict of interest in publication of this paper.
We are sorry they did not make any public statement on that issue.}

Best Regards, Arkadiy.

\smallskip
{\it (M. Tancer to A. Skopenkov,
May 12, 2021)} [A private letter.]

\smallskip
{\it (A. Skopenkov to P. Patak and M. Tancer, May 13, 2021)}

Dear Martin and Pavel,

Dear Martin, thank you for your letter.
Recall that we need a public not private statement from you.
So could you please

--- confirm that the statement you sent us  May 12, 2021 is public.

--- add either `M. Tancer' or `M. Tancer and P. Patak' at the end of the statement.

Then we would be able to publish that statement, together with my explanation why it misrepresents facts.
There is nothing wrong if you would modify your statement before confirming that the statement is public.

We strongly need this discussion to be responsible.
We do not have enough time to discuss premature ideas, whose invalidity becomes clear when their publication
(or a mental experiment of publication) is suggested.
So if the statement you sent us  May 12, 2021 is not public, the best way is to treat it as non-existent.


Therefore I inform you that each of us can possibly publish any letter on this subject starting from this letter.
After presenting the letters we can possibly write whether we agree or disagree, and/or give explanations.
If a part of such a public discussion would become obsolete, we could delete that part (only) by our mutual consent.

Such a public discussion, although very useful, would require much effort.
So let us find a way to avoid it.
E.g. if you would set up a reasonable deadline for arxiv publication of update of arXiv:1904.02404v3,
then we would be willing to postpone the arxiv publication of our paper so that it would appear after your update.
(Actually, Eugene and I worked on our paper very slowly in the hope that such an update would be available before our paper will be ready for arxiv submission.)
The deadline being reasonable means that this postponing would not obstruct too much the progress of science.
Publication of our paper after an update of yours would make the question (of conflict of interest) void.
If you want, we can discuss by skype / zoom this or other propositions.

Best, Arkadiy.

\smallskip
{\it (M. Tancer to A. Skopenkov, May 18, 2021)}\footnote{In July 2021 M. Tancer requested removal of the current paper from arxiv (arXiv:2106.14010) claiming that publication of the May 18 letter forms a copyright infringement.
This claim is unjust because
\newline
$\bullet$ the May 18 letter was sent in a reply to `Therefore I inform you that each of us can possibly publish any letter on this subject starting from this letter' of the May 13 letter;
\newline
$\bullet$ the May 18 letter did not mention either that `this letter is private' or that `this letter is copyrighted'.
\newline
I am sorry that M. Tancer did the above instead of
\newline
$\bullet$ consenting to my suggestion of removing most of the May 18 letter by mutual consent (see the June 5 letter above);
\newline
$\bullet$ asking to remove the entire May 18 letter in reply to versions of the current paper containing this letter, which were sent to M. Tancer and P. Patak on May 27 and on June 5, each time asking for remarks.
\newline
However, most of the May 18 letter is not relevant to the discussion.
So we were glad to incorporate here all requests made by M. Tancer on deleting this letter.}

\smallskip
{\it (A. Skopenkov to P. Patak and M. Tancer, May 27, 2021)}

Dear Martin and Pavel,

Attached please find the update of our paper.
We would be grateful for any remarks.

We deleted Remark 2e of the previous version.
We would be glad to restore that remark if you allow us to do so.
That remark praised your paper \cite{PT3} for presenting an argument which can easily be turned into a proof of  a certain result stated in my paper in a weaker form and only with a hint to a proof.
We deleted that remark because it used information that I received from you while I was a referee of your paper.

Recall that Martin stated that there is a `conflict of interests' in his letters of February and May, 2021.
The reasons why we need the discussion to be public are explained in my letter of May 13, 2021.

In my opinion, updating the arxiv version when the current arxiv version is under review, is a friendly action towards the referees, the Editors and math community (in case the update is essential enough).
This action {\it allows} the referees' and the Editors' decision on the paper to be more informed, if they are inclined to read or browse the update.
This action does not {\it force} the referees and the Editors to read the update
if they are not inclined to do so.


If you update your paper on arXiv, then we update our paper by replacing in Remark \ref{r:earlier} references to your paper with references to the update of your paper.
This would make the question of conflict of interest void.

Best, Arkadiy.

\smallskip
{\it (M. Tancer to A. Skopenkov, May 28, 2021)}
[A letter starting with `this letter is private' after `Therefore I inform you that each of us can possibly publish any letter on this subject starting from this letter' of May 13 letter.]

\smallskip
{\it (M. Tancer to A. Skopenkov, May 31, 2021)}


A. Skopenkov was a referee when we submitted \cite{PT3} to a journal.
We know this because he contacted us directly and he requested regular online
meetings while we were explaining the contents of our paper. In short, he
required some modifications that we could not accept as the authors which
yielded his recommendation to reject the paper. This may of course be a
legitimate approach. However, we believe that the fact
that he immediately started to work on the same topic puts him into the conflict of interests.
It therefore raises the question whether his intentions to reject the
paper were honest, or whether he wanted to promote his own work.

In addition, the further reason why A. Skopenkov is in conflict of
interests is that he had access to an extra information about paper and
about our methodology beyond the publicly available version of our paper
\cite{PT3}. He requested such an additional information as a referee. In
particular he had access to numerous intermediate revisions of our text
when we tried to rewrite several proofs, definitions, etc. in order to try
to satisfy the demands of the referee. The contents of this paper builds
on this extra information, at least partially.

\smallskip
{\it (A. Skopenkov to P. Patak and M. Tancer, June 5, 2021)}

Dear Martin and Pavel,

Attached please find the update of our paper.

(1) Let me start with a suggestion (approved by Eugene).
If you feel that some phrases (before Remark \ref{r:lett}) describe your contribution in a misleading way,
please list these phrases, explain what is wrong and/or suggest alternative phrases.
Our paper from its first version gave all the credit for main results to \cite{PT3}.
If you feel that besides that, some credit for {\it exposition} should be given to you, please name particular places (statements, elements of proofs etc.) for doing that.
If there would be many such places, we would be glad to invite you to be coauthors of this paper, on the condition that this would not lead to a significant delay and to making the text less accessible.

{\it Our intention is to facilitate progress in science by presenting without delay a short exposition of your beautiful proof, acknowledging your priority.}
(In May 2020 I hoped that this intention could be implemented within your own paper.)
I would be grateful if you could describe your intention.

(2) I am sorry that your letter misrepresents facts, and so necessarily arrives at a wrong conclusion.
Namely, the following passages are wrong: `requested', `we were explaining the contents of our paper'
`he required some modifications that we could not accept as the authors which yielded his recommendation
to reject the paper'\footnote{\label{f:reject} [added in 2022] Formally, this is correct and (as the authors write) legitimate.
However, since the May, 31 letter questions my `intentions to reject the paper', the report is presented in Remark \ref{r:report}, so that a reader can see that the reliability standards of this report agree with my reliability standards for other papers as exposed in \cite{Sk21d}.}, `he immediately started to work on the same topic'\footnote{[added in 2022] This (wrong and unjustified) guess is erroneously called `the fact' in the May, 31 letter. The authors misrepresent my wish to help them as my starting to work on the same topic.
E.g. suggestions on the proof of Theorem \ref{t:patan} were realized in this paper a year after the authors refused to fully incorporate them (see also Remark \ref{r:earlier}); suggestions from Remark \ref{r:report}.8 were realized as \cite[Proposition 2.5.1.RI Remark 2.5.2.b]{KS21e} more than a year after the authors refused to incorporate them (see also \cite{PT19}).}, `He requested such an additional information as a referee'.
Indeed,

$\bullet$ The discussions (meetings) were suggested not requested. In these discussions I presented critical remarks to the paper and suggestions on how to work on them. In particular, I did not request any extra information about \cite{PT3} and about the methodology beyond \cite{PT3}. See also my letter to the Editor in Remark \ref{r:aut}.


$\bullet$ My rejection recommendation was based not on the authors' refusal to accept my suggestions, but on specific critical remarks (1)-(8) of Remark \ref{r:report} and on the poor work of the authors on these remarks.\footnote{[added in 2022] See footnote \ref{f:reject}.}

$\bullet$ Eugene and I started to work on the current paper in November 2020.
This is long since you informed me in July 2020 that you are not going to incorporate my suggestions
(in their full extent leading to a short well-structured exposition presented here [added in 2022: and in \cite{KS21e}]).
Eugene and I worked on an {\it expository} paper giving all the credit to \cite{PT3}.

I am sorry that your letter presents no specific examples of `extra information' and `the contents of this paper'
mentioned in the phrases {\it `he had access to an extra information about paper and about our methodology beyond
the publicly available version of our paper'} and {\it `The contents of this paper builds on this extra information'}.
So an outside observer can only conclude that there are no such examples, and your conclusions are wrong.
If you would like to present such examples, see (1).

(3) In my opinion,
if some version of a paper requires explanations of its contents (even to a colleague working in a close area) and conceals some methodology,
then the version should not be published in a refereed journal.
If the authors discover such drawbacks in the frame of a refereeing process, then the best way is to withdraw the paper (and possibly resubmit a revision when it is ready).
Please let me know if you have a different opinion.

(4) In reply to your letter of May 28, recall that this discussion is public, see my letter of May 13.
So, if I receive any letter on this subject described to be a private letter, I'll have to delete it unread (to avoid confusion).
Recall that I suggested avoiding a public discussion by having a (private) skype / zoom discussion.
Also, I would be glad to discuss any other subject in any form you like
(until that subject would require a public discussion for reasons explained in my letter of May 13).

Best regards, Arkadiy.\footnote{Arkadiy suggested to delete this by mutual consent:
\newline
PS Could you please approve or disapprove my suggestions to delete by mutual consent
some material including this footnote.
\newline
PSS There is nothing wrong in delaying your answer.
I do not consider a 5-days reply to be a delay.
I suggest omitting publication of this PS and of PS from your May 18 letter by our mutual consent
(see my letter of May 13), so as not to flood the main topic of the discussion.}

\smallskip
{\it (A. Skopenkov to P. Patak and M. Tancer, February 26, 2022)}

Dear Martin and Pavel,

Attached please find the project of the update of arXiv:2106.14010.
We would be grateful for any remarks.
Some annoying flaws in our and your papers survived to arXiv publication because of our lack of exchange on these papers.  

We would be glad to remove Remarks \ref{r:aut} and \ref{r:report} if you publicly state something like
{\it `The reliability standards of A. Skopenkov's report to \cite{PT3}
agree with his reliability standards for other papers as exposed in \cite{Sk21d}.
So, as opposed to our letter of May, 31, 2021, there is no reason to believe that his reasons for recommending rejection were dishonest.}

We would be glad to remove Remark \ref{r:lett} if you publicly withdraw your claim (shown to be wrong by Remark \ref{r:lett}) that there is a conflict of interest in publication of this paper, see (1) and (2) of my May, 8, 2021 letter.

Best regards, Arkadiy.
\normalsize
\end{remark}
 
\small

\begin{remark}[A. Skopenkov. A letter to an Editor of June 2020]\label{r:aut}
Dear ...,

Hope you are fine and healthy.

The Patak-Tancer paper submitted to ... has high potential but needs a thorough revision.
Instead of thoroughly justifying these statements in a formal report I suggested to the authors (and they accepted)
a more effective way of discussing specific remarks and suggestions directly with the authors and helping them to prepare a revision.  (Cf. arXiv:2003.12285v1, Remark 3a:
Journal publications practically rule the mathematical world.
So writing a referee report on a paper is a responsible task involving double-checking.
In this time-consuming form it is much harder to help the author than via informal discussions.).

The result of these discussions would be my final report (hopefully with acceptance recommendation)
together with a list of my  suggestions (hopefully minor) which the authors intentionally did not realize.
The authors are making good progress on my suggestions.
There are points where we disagree, but we are likely to reach a compromise.
For this compromise we would need an opinion of yours (or of another referee) on those points,
or just a report from another referee.
We plan to send a letter describing our question in about a week's time.
Since (I am sorry) I missed the deadline 30 May 2020, I decided to inform you right away on our plans.

I am sorry if this non-standard approach, however more effective than a formal report,
will take more of your time than you intend to spend.
I am willing to be as consistent with the standard system as not to do ineffective work.
E.g. please let me know if I should upload this letter as a `letter to the Editor' to the Editorial system.

Best wishes, Arkadiy.
\end{remark}
 
 
\begin{remark}[A. Skopenkov. A letter to an Editor of August 2020]\label{r:report}
Here I present a public report to \cite{PT3} (this report is prepared for some journal).
Although this report is public, I only plan to publish it, together with objections to it, if some objections would be raised outside our private discussions with the authors.\footnote{The report is kept here
as a justification of footnote \ref{f:reject}, although most of the criticism is not relevant to \cite{PT19}.
In this report references are updated and grammar typos are corrected, but no other changes are made.}
I will not answer to any objections which are not public.
Motivations for that are presented in `Work on critical remarks' below.

I invited the authors to add to this public text their objections (if any).
They disagreed.
Still, the report below is too detailed at some places to rule out potential objections
(which appeared in our discussions).

I am so sorry I have to spend time on writing and double-checking this text rather than on informal private skype discussions helping the authors to improve the paper.

\smallskip
{\bf Recommendation.}

I recommend to reject the paper in its current form (i.e. in the submitted form), and to invite the resubmission of a shorter paper containing bright results (Theorem 5, its corollary Theorem 6 and Corollary 7; the shorter paper can well have a short remark on the van Kampen obstruction).

Presumably  problems from my critical remarks below could be fixed.
However, in my opinion our work on these remarks in June-July (but not our work in April-May) described below
shows that the amount of time and efforts required from the referee and the Editors to bring the part of the paper involving the van Kampen obstruction to publishable form is not worth the corresponding results of the paper.
This justifies the rejection recommendation.

The authors' work on my remarks in April-May (not presented here) was very fruitful and so justifies the invitation recommendation.

\smallskip
{\bf Critical remarks.}

Unless otherwise indicated, I presented these remarks (some of them in a less explicit form) in our April-May skype discussions with the authors.
I use references as listed in the paper and the following reference.

[BKK] {\it M. Bestvina, M. Kapovich and B. Kleiner. Van Kampen's embedding obstruction for discrete groups, Invent. Math. 150 (2002) 219--235. arXiv:math/0010141.}


\smallskip
(1)  I suggest to postpone technical results and definitions to later (sub)sections and bring bright results to earlier (sub)sections.
Most important example
of this is to present the (generalized) van Kampen obstruction after the results not involving the obstruction (or even to extract all considerations of the obstruction to a separate paper to be submitted to a less rated journal).
E.g. to move \S1.1 after \S1.3 (or maybe even later), \S2.2 after \S2.3 (or maybe even later).
See (2)-(7) below.

\smallskip
(2) We agree with the authors that bright results of this paper are Theorem 5, its corollary Theorem 6 and Corollary 7.
Statements of these results do not use the van Kampen obstruction.
Proofs of these results are non-trivial and interesting.
The proofs do not use (or can easily be simplified not to use) the van Kampen obstruction.
They need much less technical (generalization of) $\Z_2$-valued van Kampen invariant, cf.  arXiv:1805.10237, \S1.4.
Thus is misleading to call these bright results applications of the van Kampen obstruction, i.e of Theorem 1.
See \cite[Remark 1.1.b]{Sk21d}.

\smallskip
(3) Previous research suggests that the van Kampen obstruction is useful for algorithmic applications [MTW11], but
only provides unnecessary sophistication for proofs of bright results. Compare the following papers:

\emph{B. Ummel. The product of nonplanar complexes does not imbed in 4-space, Trans. Amer. Math. Soc., 242 (1978) 319--328.}

\emph{M. Skopenkov. Embedding products of graphs into Euclidean spaces,
Fund. Math. 179 (2003),~191--198, arXiv:0808.1199.}

The introduction does not present  algorithmic applications except minor Theorem 10 (cf. (2) above).
In my opinion, bright results and their proofs are potentially more useful than technical versions of known constructions which so far did not yield any bright results.

\smallskip
(4) There are many versions of the classical van Kampen obstruction.
See [Mel09], [BKK] and the following paper.\footnote{Possibly I did not mention [BKK, RS] in April-May discussions with the authors.
However, if the authors consider their generalization of the van Kampen obstruction as important as to be presented at the beginning of the introduction, then it would be natural to search earlier publications on such  generalizations before submission, or at least after our April-May discussions (and before presenting the version justifying the authors' wish to present van Kampen obstruction at the beginning of the introduction, see `Work on critical remarks' below).}

[RS] \emph{D. Repov\v s and A. B. Skopenkov. A deleted product criterion for approximability of a map by embeddings, Topol. Appl. 1998. 87 P.~1-19.}

Not citing these references is a negligible drawback if the authors present their yet another version of the van Kampen obstruction outside the introduction.
Since the authors consider their version as important as to be presented at the beginning of the introduction, not citing [BKK, RS] is a more serious drawback.

Yet another new version (Theorem 1) is not interesting enough, unless it yielded some bright results whose statements do not involve technical description of this version.
According to the introduction, Theorem 1 and Corollary 3 did not yield such results.

\smallskip
(5) No motivations for technical description of the van Kampen obstruction in \S1.1 is given.
Statements of the main results are started in \S1.1 with
{\it `We need some technical preliminaries. Also, for some notions we will not give a precise
definition yet as we would need too many preliminaries in the introduction, but all notions are explained
in Section 2.'}
After those sentences (and a meaningless formula, see (6) below) a reader is likely to put away the paper and not to reach the bright results (Theorem 5, Theorem 6 and Corollary 7).


\smallskip
(6) The van Kampen obstruction is required for the statements of Theorem 1, Corollary 3, Theorem 4 and for the proofs (not for the statements) of Proposition 8 and Theorem 10.
These statements and proofs are not rigorous because they rely on definitions involving meaningless formula
$$(*)\qquad\xi(\sigma\times\tau)=(-1)^k\xi(\tau\times\sigma).$$
This formula of \S1.1 is meaningless (according to common definition of chains with integer coefficients, see textbooks by Fomenko-Fuchs, Hatcher  or

\url{https://en.wikipedia.org/wiki/Simplicial_homology#Definition}).


Indeed, no orientation on $\sigma\times\tau$ (or on $\sigma,\tau$) is chosen, so the number
$\xi(\sigma\times\tau)$ is not defined (only its absolute value is defined).\footnote{I did not mention this  particular formula in our April-May skype discussions with the authors.
However, I did mention that fixing orientation is necessary to work with $\Z$-coefficients, in relation to other part of the text, and I did refer to Remark 1.6.4 of arXiv:1805.10237. So it would be natural to correct analogous mistake in the version justifying the authors' wish to present van Kampen obstruction at the beginning of the introduction, see `Work on critical remarks' below).}
Analogously, the number $\xi(\tau\times\sigma)$ is not defined.
Observe that the sign is important in the formula (*).

Presumably the authors use some non-specified non-standard orientation convention which lead them to
skew-symmetric cochains on $\t K$ rather than symmetric cochains on $\t K$ isomorphic to ordinary cochains on
$\t K/\Z_2$ as in [Sha57, \S3] and in the survey [Sko08, \S4].
No motivation for this sophistication is presented.

\smallskip
(7) The construction of \S1.1 is not invariant, i.e. it uses (co)chains, not (co)cycles and (co)homology classes;
this makes it less valuable from a theoretical point of view.



\smallskip
(8) {\it A possible compromise} is to state mod 2 version of Theorem 1 in the following user-friendly way similar to the Sarkaria's definition of the obstructor complex, see e.g. [BKK, Definition 4 in \S2].
(There is perhaps an analogous statement over integers,  with analogous proof; hopefully this invariant statement needs neither technical condition (H) nor (H').)

{\bf Theorem 1.}
Let $K$ be a $k$-complex, $M$ a PL $2k$-manifold, and $f:|K|\to M$ be a generic map homotopic to an embedding.
Assume that

(H') the restriction of $f$ the $(k-1)$-skeleton of $K$ is null-homotopic.

Denote $\t K:=\cup\{\sigma\times\tau\in K\times K\ :\ \sigma\cap\tau=\emptyset\}$.
Let $C$ be any $2k$-cycle modulo 2 in $\t K/\Z_2$ whose preimage in $\t K$ equals to
$\sum_j(\alpha_j\times\beta_j + \beta_j\times\alpha_j)$
for some $k$-cycles $\alpha_j,\beta_j$ in $K$.
Then
$$\sum_j |f\alpha_j\cap f\beta_j|
\equiv
\sum\limits_{ \{\sigma,\tau\} \in C} |g\sigma\cap g\tau|
\mod2.$$
(The later sum equals to the value on $C$ of the classical van Kampen obstruction of $K$.)

{\bf Remark 2.}
In more theoretical terms the conclusion of Theorem 1 means (at least for $(k-1)$-connected $M$) that the classical van Kampen obstruction of $K$ has to be the image of the mod 2 intersection form of $M$ under the composition
$$\mbox{Hom}_{sym}(H^k(M)\otimes H^k(M);\Z_2)\to
(H^k(M)\otimes H^k(M))_{sym}\overset{\kappa}\to H^{2k}(M\times M)_{sym}\overset{(f\times f)^*}\to $$
$$\to H^{2k}(K\times K)_{sym}\overset{r}\to H^{2k}_{sym}(\t K)\overset{\cong}\to H^{2k}(\t K/\Z_2).$$
Here coefficients $\Z_2$ are omitted, $(H^k(M)\otimes H^k(M))_{sym}$ is the subgroup of $H^k(M)\otimes H^k(M)$ generated by elements $\alpha\otimes\alpha$ and $\alpha\otimes\beta+\beta\otimes\alpha$,
the homomorphism $\kappa$ comes from the K\"unneth theorem and $r$ is the restriction.

\smallskip
{\bf Work on critical remarks.}

Upon my suggestion and agreement of the authors and the Editor, in April-May the authors and I discussed my critical remarks and suggestions by skype.
My identity as a referee was revealed upon my wish.
Our discussions  were very effective, cf. \cite[Remark 1.4.a]{Sk21d}.
The authors greatly improved the text.


For the parts of the paper involving the van Kampen obstruction
we did not agree and we wanted to ask the Editor (and possibly another referee) to step in.
Namely,

$\bullet$ we agreed that Theorem 4, Proposition 8 and Theorem 10 are less important than other results of the introduction;

$\bullet$ we only disagreed on the comparative value of Theorem 1 (+Corollary 3) versus
the amount of motivations and technical details required for its statement.


The authors prepared updated version of the paper
to illustrate the differences and to justify their point of view.
For justification of my rejection recommendation it suffices to know that the above critical remarks on the submitted version are also pertinent for the updated version, although I presented these remarks
in our April-May skype discussions (with  exceptions mentioned in the footnotes above, which are irrelevant for the rejection recommendation).
Since it was not so easy for the authors to prepare an update properly motivating and rigorously defining the van Kampen obstruction, it would be hard for a reader to understand the motivation and the definition.

We also planned a common letter to the editor containing descriptions of our different opinions.
In preparation of this letter there appeared less-responsible and/or unjustified passages (i.e. those which would not stand the test of making them public).
So I required a public discussion where none of our letters is changed after its writing (the resulting dialogue to be sent to the Editor, but not to be put it on the internet unless we will be forced to).
I had to make our letter public in order to

$\bullet$ make discussion of our different opinions more responsible;

$\bullet$ obstruct unjustified criticism of my report (if such a criticism would appear);

$\bullet$ submit to justified criticism of my report (if such a criticism would appear).

The authors disagreed.
Instead, they agreed to incorporate my April-May suggestions on the van Kampen obstruction (at least in a compromise form we considered in April-May).
Besides my repeating critical remarks in written form and impossibility of further less-responsible discussion,
there could be other reasons for the authors' change of mind.
Whatever the authors' reasons, I spent much more time on convincing them to rigorously define the van Kampen obstruction and to properly motivate it (to justify its place in the paper), than time sufficient to make corresponding changes  myself.
\end{remark}

\normalsize


{\it Books, surveys and expository papers in this list are marked by the stars.}


\begin{thebibliography}{RSS95}

\UseRawInputEncoding

\newcommand{\aate}{\bibitem[AA38]{AA38} \emph{A. Adrian Albert}. Symmetric and alternate matrices in an arbitrary field, I. Trans. Amer. Math. Soc., (1938) 43(3):386--436.}

\newcommand{\abc}{\bibitem[ABC+]{ABC+} * \emph{M. Atiyah, A. Borel, G. J. Chaitin, D. Friedan, J. Glimm, J. J. Gray, M. W. Hirsch, S. MacLane, B. B. Mandelbrot, D. Ruelle, A. Schwarz, K. Uhlenbeck, R. Thom, E. Witten, C.  Zeeman.} Responses to ``Theoretical Mathematics: Toward a cultural synthesis of mathematics and theoretical physics'', by A. Jaffe and F. Quinn. Bull. Am. Math. Soc. 30 (1994) 178--207. arXiv:math/9404229.}

\newcommand{\abgmns}{\bibitem[ABM+]{ABM+} * \emph{E. Alkin, E. Bordacheva, A. Miroshnikov, O. Nikitenko, A. Skopenkov,} Invariants of almost embeddings of graphs in the plane: results and problems, arXiv:2408.06392.}

\newcommand{\abms}{\bibitem[ABM+]{ABM+} * \emph{Э. Алкин, Е. Бордачева, А. Мирошников, А. Скопенков,} Инварианты почти вложений графов в плоскость, arXiv:2410.09860.}

\newcommand{\ams}{\bibitem[AMS]{AMS} * \emph{Э. Алкин, А. Мирошников, А. Скопенков,} Инварианты почти вложений графов в плоскость, arXiv:2410.09860v2.}

\newcommand{\adnt}{\bibitem[Ad93]{Ad93} * \emph{M. Adachi}. Embeddings and Immersions. Amer. Math.
Soc., 1993. (Transl. of Math. Monographs; V.~124).}

\newcommand{\adoe}{\bibitem[Ad18]{Ad18} {\it K. Adiprasito,} Combinatorial Lefschetz theorems beyond positivity, arXiv:1812.10454v4.}

\newcommand{\adnsv}{\bibitem[ADN+]{ADN+} * \emph{E. Alkin, S. Dzhenzher, O. Nikitenko, A. Skopenkov, A. Voropaev.}
Cycles in graphs and in hypergraphs: results and problems, arXiv:2308.05175.}

\newcommand{\agles}{\bibitem[AGL]{AGL86} Mathematical Economics,  ed. by A. Ambrosetti, F. Gori, R. Lucchetti,
Lect. Notes Math. 1330, Springer, 1986.}


\newcommand{\akzz}{\bibitem[Ak00]{Ak00} * \emph{П. М. Ахметьев.} Вложения компактов, стабильные
гомотопические группы сфер и теория особенностей, Успехи Мат. Наук.  2000. 55:3. C.~3-62.}

\newcommand{\akoe}{\bibitem[AK19]{AK19} \emph{S. Avvakumov, R. Karasev.} Envy-free division using mapping degree,
Mathematika, 67:1 (2020), 36--53. arXiv:1907.11183.}

\newcommand{\akto}{\bibitem[AK21]{AK21} \emph{G. Arone and V. Krushkal.}
Embedding obstructions in $\R^d$ from the Goodwillie-Weiss calculus and Whitney disks, 	Asian J. Math. 27 (2023), 135--186. arXiv:2101.10995. }

\newcommand{\akm}{\bibitem[AKM]{AKM} \emph{M. Abrahamsen, L. Kleist and T. Miltzow.}
Geometric Embeddability of Complexes is $\exists\mathbb R$-complete. arXiv:2108.02585.}

\newcommand{\aksoe}{\bibitem[AKS]{AKS} \emph{S. Avvakumov, R. Karasev and A. Skopenkov.} Stronger counterexamples to the topological Tverberg conjecture, Combinatorica, 43 (2023), 717--727. arXiv:1908.08731.}


\newcommand{\akuoe}{\bibitem[AKu19]{AKu19} \emph{S. Avvakumov, S. Kudrya.}
Vanishing of all equivariant obstructions and the mapping degree.
Discr. Comp. Geom., 66:3 (2021) 1202--1216. arXiv:1910.12628.}

\newcommand{\aktf}{\bibitem[Ak]{Ak} \emph{Д. Акимов,} Нечетность суммы чисел оборотов для почти вложения графа $K_{3,2}$.}

\newcommand{\alto}{\bibitem[Al22]{Al22} \emph{E. Alkin,}
Hardness of almost embedding simplicial complexes in $\R^d$, II. arXiv:2206.13486}

\newcommand{\amtf}{\bibitem[AM25]{AM25} \emph{E. Alkin, A. Miroshnikov,} On winding numbers of almost embeddings of $K_4$ in the plane, Math. Notes, to appear, arXiv:2501.15642.}

\newcommand{\amsw}{\bibitem[AMS+]{AMSW} \emph{S. Avvakumov, I. Mabillard, A. Skopenkov and U. Wagner.}
Eliminating Higher-Multiplicity Intersections, III. Codimension 2, Israel J. Math. 245 (2021) 501--534.  arxiv:1511.03501.}


\newcommand{\anzt}{\bibitem[An03]{An03} * \emph{Д. В. Аносов.} Отображения окружности, векторные поля и их применения. М: МЦНМО, 2003.}

\newcommand{\anstf}{\bibitem[ANS]{ANS} \emph{E. Alkin, O. Nikitenko, A. Skopenkov,} Homotopy classification of closed polygonal lines: results and problems, arXiv:2508.16287.}

\newcommand{\arnf}{\bibitem[Ar95]{Ar95} * \emph{V. I. Arnold,}  Topological invariants of plane curves and caustics, University Lecture Series, Vol. 5, Amer. Math. Soc., Providence, RI, 1995.}

\newcommand{\arszo}{\bibitem[ARS01]{ARS01} \emph{P. Akhmetiev, D. Repov\v s and A. Skopenkov},
Embedding products of low-dimensional manifolds in $\R^m$, Topol. Appl. 113 (2001), 7--12.}

\newcommand{\arszt}{\bibitem[ARS02]{ARS02} \emph{P. Akhmetiev, D. Repovs and A. Skopenkov.} Obstructions to approximating maps of $n$-manifolds into $R^{2n}$ by embeddings, Topol. Appl., 123 (2002), 3--14.}

\newcommand{\asoed}{\bibitem[As]{As} \emph{A. Asanau,} \lowercase{A SIMPLE PROOF THAT CONNECTED SUM OF ORDERED
ORIENTED LINKS IS NOT WELL-DEFINED,} Math. Notes, to appear.}

\newcommand{\asoe}{\bibitem[As]{As} \emph{A. Asanau,} On the \lowercase{TRIPLE SELF-INTERSECTION NUMBER FOR GRAPHS IN THE PLANE,} unpublished, 2018.}

\newcommand{\avos}{\bibitem[Av14]{Av14} \emph{S. Avvakumov,} The classification of certain linked 3-manifolds in 6-space, Moscow Math. J., 16:1 (2016), 1--25. arXiv:1408.3918.}

\newcommand{\avose}{\bibitem[Av17]{Av17} \emph{S. Avvakumov,} The classification of linked 3-manifolds in 6-space, Algebraic \& Geometric Topology, 22:6 (2022) 2587--2630. arXiv:1704.06501.}



\newcommand{\bant}{\bibitem[Ba93]{Ba93} * \emph{T. Bartsch.} Topological methods for variational problems with
symmetries, Lecture Notes in Mathematics, 1560, Springer-Verlag, Berlin, 1993.}

\newcommand{\batt}{\bibitem[Ba23]{Ba23} * \emph{I. Barany.} Tverberg's theorem, a new proof. arXiv:2308.10105.}

\newcommand{\bbsn}{\bibitem[BB79]{BB} \emph{E.~G. Bajm{{\'o}}czy and I.~B{{\'a}}r{{\'a}}ny,}
\newblock On a common generalization of {B}orsuk's and {R}adon's theorem,
\newblock Acta Math.\ Acad.\ Sci.\ Hungar.\ 34:3 (1979), 347-350.}

\newcommand{\bbzos}{\bibitem[BBZ]{BBZ} * \emph{I.~B{{\'a}}r{{\'a}}ny, P.~V.~M. Blagojevi{{\'c}} and G.~M. Ziegler.} Tverberg's Theorem at 50: Extensions and Counterexamples, Notices of the Amer. Math. Soc., 63:7 (2016), 732--739.}


\newcommand{\bcm}{\bibitem[BCM]{BCM} * 13th Hilbert Problem on superpositions of functions, presented by A. Belov, A. Chilikov, I. Mitrofanov, S. Shaposhnikov and A. Skopenkov,
\url{http://www.turgor.ru/lktg/2016/5/index.htm}.}

\newcommand{\beet}{\bibitem[BE82]{BE82} * \emph{V.G. Boltyansky and V.A. Efremovich.} Intuitive Combinatorial Topology. Springer.}

\newcommand{\beetr}{\bibitem[BE82]{BE82} * \emph{В. Г. Болтянский и В. А. Ефремович.} Наглядная топология. М.:  Наука, 1982.}


\newcommand{\bfzn}{\bibitem[BF09]{BF09} \emph{K. Barnett, M. Farber}. Topology of Configuration Space of Two Particles on a Graph, I.  Algebr. Geom. Topol. 9 (2009) 593--624.	arXiv:0903.2180.}

\newcommand{\bfzof}{\bibitem[BFZ14]{BFZ14} \emph{P. V. M. Blagojevi{\'c}, F. Frick, and G. M. Ziegler,}
Tverberg plus constraints, Bull. Lond. Math. Soc. 46:5 (2014), 953-967, arXiv:1401.0690.}


\newcommand{\bfzos}{\bibitem[BFZ]{BFZ} \emph{P. V. M. Blagojevi{\'c}, F. Frick and G. M. Ziegler,}
Barycenters of Polytope Skeleta and Counterexamples to the Topological Tverberg Conjecture, via Constraints,
J. Eur. Math. Soc., 21:7 (2019) 2107-2116. arXiv:1510.07984.}


\newcommand{\bgso}{\bibitem[BG71]{BG71} J.C. Becker and H. H. Glover, {\it Note on the Embedding of Manifolds in Euclidean Space,} Proc. of the Amer. Math. Soc., 27:2 (1971) 405-410.}


\newcommand{\bgos}{\bibitem[BG16]{BG16} \emph{A. Bj\"orner and A. Goodarzi}, On Codimension one Embedding of Simplicial Complexes, in book: A Journey Through Discrete Mathematics, arXiv:1605.01240.}

\newcommand{\biet}{\bibitem[Bi83]{Bi83} * \emph{R. H. Bing.} The Geometric Topology of 3-Manifolds. Providence, R.~I. 1983. (Amer. Math. Soc. Colloq. Publ., 40).}

\newcommand{\bitz}{\bibitem[Bi20]{Bi20} * \emph{A. Bikeev.} Realizability of discs with ribbons on the M\"obius strip. Mat. Prosveschenie, 28 (2021), 150-158;
erratum to appear. arXiv:2010.15833.}

\newcommand{\bitzr}{\bibitem[Bi20]{Bi20} * \emph{А. Бикеев.} Реализуемость дисков с ленточками на ленте Мебиуса.
Мат. просвещение. Сер. 3. 28 (2021), 150--158.}

\newcommand{\bito}{\bibitem[Bi21]{Bi21} {\it A. I. Bikeev,}
Criteria for integer and modulo 2 embeddability of graphs to surfaces, arXiv:2012.12070v2.}


\newcommand{\bagos}{\bibitem[BG17]{BG17} \emph{S. Basu and S. Ghosh.} Equivariant maps related to the topological Tverberg conjecture, Homology, Homotopy and Applications 19:1 (2017) 155--170.}

\newcommand{\bkkmzof}{\bibitem[BKK]{BKK} \emph{M. Bestvina, M. Kapovich and B. Kleiner,}
Van Kampen's embedding obstruction for discrete groups, Invent. Math. 150 (2002) 219--235. arXiv:math/0010141.}

\newcommand{\bl}{\bibitem[BL]{BL} \url{https://en.wikipedia.org/wiki/Brunnian_link}}

\newcommand{\blf}{\bibitem[BL4]{BL4} Students form a 4-component Brunnian link,  \url{http://www.mccme.ru/circles/oim/foto2014/brunn4.png} (5Mb)}

\newcommand{\bmzf}{\bibitem[BM04]{BM04} \emph{Boyer, J. M. and Myrvold, W. J.} On the cutting edge: simplified $O(n)$ planarity by edge addition,  Journal of Graph Algorithms and Applications, 8:3 (2004) 241--273.}

\newcommand{\bm}{\bibitem[BM15]{BM15} \emph{I. Bogdanov and A. Matushkin.} Algebraic proofs of linear versions of the Conway--Gordon--Sachs theorem and the van Kampen--Flores theorem, arXiv:1508.03185.}


\newcommand{\bmzzn}{\bibitem[BMZ09]{BMZ09} \emph{P. V. M. Blagojevi{\'c}, B. Matschke, G. M. Ziegler,}
Optimal bounds for a colorful Tverberg-Vre\'cica type problem, Advances in Math., 226 (2011), 5198-5215, arXiv:0911.2692.}

\newcommand{\bmzof}{\bibitem[BMZ15]{BMZ15} \emph{P. V. M. Blagojevi{\'c}, B. Matschke, G. M. Ziegler,}
Optimal bounds for the colored Tverberg problem, J. Eur. Math. Soc.,  17:4 (2015) 739--754,
arXiv:0910.4987.}

\newcommand{\bpns}{\bibitem[BP97]{BP97} * \emph{R. Benedetti and C. Petronio.} Branched standard spines of 3-manifolds, Lecture Notes in Math. 1653, Springer-Verlag, Berlin-Heidelberg-New York, 1997.}

\newcommand{\brst}{\bibitem[Br72]{Br72} \emph{J. L. Bryant.} Approximating embeddings of polyhedra in codimension 3, Trans. Amer. Math. Soc., 170 (1972) 85--95.}

\newcommand{\brts}{\bibitem[Br68]{Br68} \emph{P. Bruegel,} 1568,
\url{https://en.wikipedia.org/wiki/The_Magpie_on_the_Gallows}.}


\newcommand{\bren}{\bibitem[Br82]{brown1982} * \emph{K.~S. Brown.} \newblock Cohomology of Groups. \newblock Springer-Verlag New York, 1982.}


\newcommand{\bssos}{\bibitem[BS17]{BS17} * \emph{I.~B\'{a}r\'{a}ny and P. Sober\'{o}n,} Tverberg's theorem is 50 years old: a survey, Bull. Amer. Math. Soc. (N.S.) 55:4 (2018), 459--492. arXiv:1712.06119.}

\newcommand{\bsto}{\bibitem[BS21]{BS21} * \emph{A. Buchaev and A. Skopenkov,} Simple proofs of estimations of Ramsey numbers and of discrepancy, Mat. Prosveschenie, to appear, arXiv:2107.13831.}

\newcommand{\brsnn}{\bibitem[BRS99]{BRS99} \emph{D. Repov\v s, N. Brodsky and A. B. Skopenkov.}
A classification of 3-thickenings of 2-polyhedra, Topol. Appl. 1999. 94. P.~307-314.}

\newcommand{\bsseo}{\bibitem[BSS]{BSS} \emph{I.~B\'{a}r\'{a}ny, S.~B. Shlosman, and A.~Sz{\H{u}}cs,}
\newblock On a topological generalization of a theorem of {T}verberg,
\newblock J.\ London Math.\ Soc.\ (II. Ser.) 23 (1981), 158--164.}

\newcommand{\btzs}{\bibitem[BT07]{BT07} \emph{A. Bj\"orner, M. Tancer}, Combinatorial Alexander Duality --- a Short and Elementary Proof, Discr. and Comp. Geom., 42 (2009) 586. arXiv:0710.1172.}

\newcommand{\buse}{\bibitem[Bu68]{Bu68} \emph{A. R. Butz,} Space filling curves and mathematical programming, Information and Control, 12:4 (1968) 314--330.}


\newcommand{\bz}{\bibitem[BZ16]{BZ16} * \emph{P. V. M. Blagojevi\'c and G. M. Ziegler,} Beyond the Borsuk-Ulam theorem: The topological Tverberg story, in: A Journey Through Discrete Mathematics, Eds. M. Loebl,
J. Ne\v set\v ril, R. Thomas, Springer, 2017, 273--341. arXiv:1605.07321v3.}



\newcommand{\cano}{\bibitem[Ca91]{Ca91} * \emph{D. de Caen}, The ranks of tournament matrices, Amer. Math. Monthly, 98:9 (1991) 829--831.}

\newcommand{\ca}{\bibitem[Ca]{Ca} \emph{J. Carmesin.} Embedding simply connected 2-complexes in 3-space, I-V, arXiv:1709.04642, arXiv:1709.04643, arXiv:1709.04645, arXiv:1709.04652, arXiv:1709.04659.}

\newcommand{\cfsz}{\bibitem[CF60]{CF60} \emph{P. E. Conner and E. E. Floyd}, Fixed points free involutions and equivariant maps, Bull. Amer. Math. Soc., 66 (1960) 416--441.}

\newcommand{\cfs}{\bibitem[CFS]{CFS} \emph{D. Crowley, S.C. Ferry, M. Skopenkov,} The rational classification of links of codimension $>2$, Forum Math. 26 (2014), 239--269. arXiv:1106.1455.}

\newcommand{\cget}{\bibitem[CG83]{CG83} \emph{J. H. Conway and C. M. A. Gordon},
Knots and links in spatial graphs, J. Graph Theory  7 (1983), 445--453.}

\newcommand{\cten}{\bibitem[Ch]{Ch} \emph{Chuang Tzu,} translated by H. A. Giles, Bernard Quaritch, London, 1889.}

\newcommand{\ctruku}{\bibitem[Ch]{Ch} \emph{Chuang Tzu,} translated to Russian by S. Kuchera, in: Ancient Chinese Philosophy, v. I, Mysl, Moscow, 1972.}


\newcommand{\chnn}{\bibitem[Ch99]{Ch99} * \emph{А. В. Чернавский,} Теорема Жордана.  Мат. Просвещение, 3 (1999), 142--157.}

\newcommand{\hcon}{\bibitem[HC19]{HC19} * \emph{C. Herbert Clemens.} Two-Dimensional Geometries. A Problem-Solving Approach, Amer. Math. Soc., 2019.}

\newcommand{\ckmoo}{\bibitem[CKMS]{CKMS} \emph{M. \v Cadek, M. Kr\v c\'al. J. Matou\v sek, F. Sergeraert,
L. Vok\v r\'inek, U. Wagner.} Computing all maps into a sphere, J. of the ACM, 61:3 (2014). arXiv:1105.6257.}


\newcommand{\ckmvwot}{\bibitem[CKM12+]{CKM12+} \emph{M. \v Cadek, M. Kr\v c\'al. J. Matou\v sek, L. Vok\v r\'inek, U. Wagner.} Polynomial-time computation of homotopy groups and Postnikov systems in fixed dimension, SIAM J. Comput., 43:5 (2014), 1728--1780. arXiv:1211.3093.}

\newcommand{\ckmvw}{\bibitem[CKM+]{CKM+} \emph{M. \v Cadek, M. Kr\v c\'al. J. Matou\v sek, L. Vok\v r\'inek, U. Wagner.} Extendability of continuous maps is undecidable, Discr. and Comp. Geom. 51 (2014) 24--66.
arXiv:1302.2370.}

\newcommand{\ckppt}{\bibitem[CKP+]{CKP+} \emph{E. Colin de Verdi\'ere, V. Kalu\v za, P. Pat\'ak, Z. Pat\'akov\'a and M. Tancer.} A direct proof of the strong Hanani-Tutte theorem on the projective plane. Journal of Graph Algorithms and Applications, 21:5 (2017) 939--981.}

\newcommand{\cksof}{\bibitem[CKS+]{CKS+} * New ways of weaving baskets, presented by G. Chelnokov, Yu. Kudryashov, A.Skopenkov and A. Sossinsky, \url{http://www.turgor.ru/lktg/2004/lines.en/index.htm}.}

\newcommand{\ckv}{\bibitem[CKV]{CKV} \emph{M.~{\v{C}}adek, M.~Kr\v{c}\'{a}l, and L.~Vok\v{r}\'{\i}nek.}
Algorithmic solvability of the lifting-extension problem, Discr. Comp. Geom. 57 (2017), 915--965. arXiv:1307.6444.}


\newcommand{\clr}{\bibitem[CLR]{CLR} * \emph{Т. Кормен, Ч. Лейзерсон, Р. Ривест.} Алгоритмы:
построение и анализ, МЦНМО, Москва, 1999.}

\newcommand{\clreng}{\bibitem[CLR]{CLR} * \emph{T. H. Cormen, C. E.Leiserson, R. L.Rivest, C. Stein.} Introduction to Algorithms, MIT Press, 2009.}

\newcommand{\crzfru}{\bibitem[CR]{CR} * \emph{Р. Курант, Дж. Роббинс,} Что такое математика. М.: МЦНМО, 2004.}

\newcommand{\crzfen}{\bibitem[CR]{CR} * \emph{R. Courant and H. Robbins,} What is Mathematics, Oxford Univ. Press.}

\newcommand{\crsne}{\bibitem[CRS98]{CRS98} * \emph{A. Cavicchioli, D. Repov\v s and A. B. Skopenkov.}
Open problems on graphs, arising from geometric topology, Topol. Appl. 84 (1998), 207--226.}

\newcommand{\crsos}{\bibitem[CRS']{CRS'} \emph{M. Cencelj, D. Repov\v s and M. Skopenkov,} Homotopy type of the complement of an immersion and classification of embeddings of tori, Russian Math. Surv. 62:5 (2007), 985--987,
arXiv:0803.4285.}

\newcommand{\crsot}{\bibitem[CRS]{CRS} \emph{M. Cencelj, D. Repov\v s and M. Skopenkov,}
Classification of knotted tori in the 2-metastable dimension, Mat. Sbornik, 203:11 (2012), 1654--1681.
arxiv:0811.2745.}

\newcommand{\csoo}{\bibitem[CS08]{CS08} \emph{D. Crowley and A. Skopenkov.} A classification of smooth embeddings of 4-manifolds in 7-space, II, Intern. J. Math., 22:6 (2011) 731-757, arxiv:0808.1795.}

\newcommand{\csos}{\bibitem[CS16]{CS16} \emph{D. Crowley and A. Skopenkov,} Embeddings of non-simply-connected 4-manifolds in 7-space. I. Classification modulo knots, Moscow Math. J., 21 (2021), 43--98. arXiv:1611.04738.}


\newcommand{\csoso}{\bibitem[CS16o]{CS16o} \emph{D. Crowley and A. Skopenkov,} Embeddings of non-simply-connected 4-manifolds in 7-space. II. On the smooth classification, Proc. A of the Royal Soc. of Edinburgh 152:1 (2022), 163--181. arXiv:1612.04776.}


\newcommand{\crsk}{\bibitem[CS]{CS} \emph{D. Crowley and A. Skopenkov,} Embeddings of non-simply-connected 4-manifolds in 7-space. III. Piecewise-linear classification. draft.}

\newcommand{\cutz}{\bibitem[Cu20]{Cu20} \emph{C. Culter,} Cantor sets are not tangent homogeneous,
Topol. Appl. 271 (2020) 1--9.}


\newcommand{\dies}{\bibitem[Di87]{Di} * \emph{T. tom Dieck,} Transformation groups, Studies in Mathematics, vol. 8, Walter de Gruyter, Berlin, 1987.}

\newcommand{\dize}{\bibitem[Di08]{Di08} * \emph{T. tom Dieck,} Algebraic topology, EMS Textbooks in Mathematics, 
EMS, Z\"urich, 2008.}

\newcommand{\dent}{\bibitem[De93]{De93}  \emph{T.K. Dey.} On counting triangulations in $d$-dimensions. Comput. Geom.  3:6 (1993) 315--325.}

\newcommand{\denf}{\bibitem[DE94]{DE94}  \emph{T.K. Dey and H. Edelsbrunner.} Counting triangle crossings and halving planes, Discrete Comput. Geom. 12 (1994), 281--289.}

\newcommand{\dgn}{\bibitem[DGN+]{DGN+} * S. Dzhenzher, T. Garaev, O. Nikitenko, A. Petukhov, A. Skopenkov, A. Voropaev, Low rank matrix completion and realization of graphs: results and problems, arXiv:2501.13935.}

\newcommand{\dgnr}{\bibitem[DGN+]{DGN+} * Минимизация ранга восполнением матриц, представляли А. Воропаев, Т. Гараев, С. Дженжер, О. Никитенко, А. Петухов и А. Скопенков, \url{https://www.mccme.ru/circles/oim/netflix_rus.pdf}.}

\newcommand{\dstt}{\bibitem[DS22]{DS22}  \emph{S. Dzhenzher and A. Skopenkov,} A quadratic estimation for the K\"uhnel conjecture on embeddings, arXiv:2208.04188.}

\newcommand{\botf}{\bibitem[Dz25]{Dz25} \emph{E. Dzhenzher,} Symmetric 1-cycles in the deleted product of a graph, Topol. Appl. (2025) 109277.}



\newcommand{\embo}{\bibitem[Eb]{Eb} * \url{http://www.map.mpim-bonn.mpg.de/Embeddings_of_manifolds_with_boundary:_classification}}

\newcommand{\embe}{\bibitem[Em]{Em} * \url{http://www.map.mpim-bonn.mpg.de/Embedding_(simple_definition)}}

\newcommand{\ers}{\bibitem[ERS]{ERS} * Invariants of graph drawings in the plane, presented by A. Enne, A. Ryabichev, A. Skopenkov and T. Zaitsev, \url{http://www.turgor.ru/lktg/2017/6/index.htm}}



\newcommand{\feto}{\bibitem[Fe21]{Fe21} \emph{M. Fedorov.} A description of values of Seifert form for punctured $n$-manifolds in $(2n-1)$-space, arXiv:2107.02541.}

\newcommand{\ffen}{\bibitem[FF89]{FF89} * \emph{А. Т. Фоменко и Д. Б. Фукс.} Курс гомотопической топологии. М.: Наука, 1989.}

\newcommand{\ffene}{\bibitem[FF89]{FF89} * \emph{A.T. Fomenko and D.B. Fuchs.} Homotopical Topology, Springer, 2016.}


\newcommand{\fhzo}{\bibitem[FH10]{FH10}  \emph{M. Farber, E. Hanbury}. Topology of Configuration Space of Two Particles on a Graph, II. Algebr. Geom. Topol. 10 (2010) 2203--2227. arXiv:1005.2300.}


\newcommand{\fkosc}{\bibitem[FK17]{FK17} \emph{R. Fulek, J. Kyn{\v{c}}l,} Counterexample to an Extension of the Hanani-Tutte Theorem on the Surface of Genus 4, Combinatorica, 39 (2019) 1267--1279, arXiv:1709.00508.}

\newcommand{\fkos}{\bibitem[FK17]{FK17} \emph{R. Fulek, J. Kyn{\v{c}}l,} Hanani-Tutte for approximating maps of graphs, arXiv:1705.05243.}

\newcommand{\fkon}{\bibitem[FK19]{FK19} \emph{R. Fulek, J. Kyn{\v{c}}l,}
$\Z_2$-genus of graphs and minimum rank of partial symmetric matrices,
35th Intern. Symp. on Comp. Geom. (SoCG 2019), Article No. 39; pp. 39:1--39:16, \linebreak
\url{https://drops.dagstuhl.de/opus/volltexte/2019/10443/pdf/LIPIcs-SoCG-2019-39.pdf}.
We refer to numbering in arXiv version: arXiv:1903.08637.}

\newcommand{\fktnf}{\bibitem[FKT]{FKT} \emph{M. H. Freedman, V. S. Krushkal and P. Teichner.} Van Kampen's
embedding obstruction is incomplete for 2-complexes in~$\R^4$, Math. Res. Letters. 1994. 1. P.~167-176.}

\newcommand{\fltf}{\bibitem[Fl34]{Fl34} \emph{A. Flores}, \"Uber $n$-dimensionale Komplexe die im $E^{2n+1}$ absolut selbstverschlungen sind, Ergeb. Math. Koll. 6 (1934) 4--7.}

\newcommand{\fnzn}{\bibitem[FN09]{FN09} \emph{T. Fleming and R. Nikkuni,} Homotopy on spatial graphs and the Sato-Levine invariant, Trans. Amer. Math. Soc. 361:4 (2009), 1885--1902.}

\newcommand{\fo}{\bibitem[Fo]{Fo} * \emph{L. Fortnow.} Time for Computer Science to Grow Up,  \url{https://people.cs.uchicago.edu/~fortnow/papers/growup.pdf}.}

\newcommand{\fozf}{\bibitem[Fo04]{Fo04} * \emph{R. Fokkink.} A forgotten mathematician, Eur. Math. Soc. Newsletter 52 (2004) 9--14.}


\newcommand{\fpstz}{\bibitem[FPS]{FPS} \emph{R. Fulek, M.J. Pelsmajer and M. Schaefer.}
Strong Hanani-Tutte for the Torus, arXiv:2009.01683.}

\newcommand{\frse}{\bibitem[Fr78]{Fr78} \emph{M. Freedman,} Quadruple points of 3-manifolds in $S^4$, Comment. Math. Helv. 53 (1978), 385-394.}

\newcommand{\fres}{\bibitem[FR86]{FR86} \emph{R. Fenn, D. Rolfsen.}
Spheres may link homotopically in 4-space, J. London Math. Soc. 34 (1986) 177-184.}

\newcommand{\frofea}{\bibitem[Fr15']{Fr15'} \emph{F. Frick}, Counterexamples to the topological Tverberg conjecture, arXiv:1502.00947v1.}


\newcommand{\frof}{\bibitem[Fr15]{Fr15} \emph{F. Frick}, Counterexamples to the topological Tverberg conjecture,
Oberwolfach reports, 12:1 (2015), 318--321. arXiv:1502.00947.}

\newcommand{\fros}{\bibitem[Fr17]{Fr17} \emph{F. Frick}, O\lowercase{N AFFINE TVERBERG-TYPE RESULTS WITHOUT CONTINUOUS GENERALIZATION}, arXiv:1702.05466}


\newcommand{\fstz}{\bibitem[FS20]{FS20} \emph{F. Frick and P. Sober\'on}, The topological Tverberg problem beyond prime powers, arXiv:2005.05251.}

\newcommand{\ftss}{\bibitem[FT77]{FT77} \emph{R. Fenn, P. Taylor,} Introducing doodles, pp. 37-43
in: Topology of Low-Dimensional Manifolds, Proceedings of the Second Sussex Conference, 1977,
Ed. R. Fenn, V. 722 of Lecture Notes in Math.}

\newcommand{\fvto}{\bibitem[FV21]{FV21} \emph{M. Filakovsk\'y, L. Vok\v r\'inek.} Computing homotopy classes for diagrams, Discr. Comp. Geom. 70 (2023), 866--920. arXiv:2104.10152.}

\newcommand{\fwz}{\bibitem[FWZ]{FWZ} \emph{M. Filakovsk\'y, U. Wagner, S. Zhechev.} Embeddability of simplicial complexes is undecidable. Oberwolfach reports, to appear.}

\newcommand{\fwztz}{\bibitem[FWZ]{FWZ} \emph{M. Filakovsk\'y, U. Wagner, S. Zhechev.} Embeddability of simplicial complexes is undecidable. Proceedings of the 2020 ACM-SIAM Symposium on Discrete Algorithms.}



\newcommand{\ga}{\bibitem[GA]{GA} * \url{https://en.wikipedia.org/wiki/Galactic_algorithm}}

\newcommand{\gatt}{\bibitem[Ga23]{Ga23} \emph{T. Garaev}, On drawing $K_5$ minus an edge in the plane, arXiv:2303.14503.}

\newcommand{\gdikrse}{\bibitem[GDI]{GDI} * {\it A. Chernov, A. Daynyak, A. Glibichuk, M. Ilyinskiy, A. Kupavskiy, A. Raigorodskiy and A. Skopenkov,} Elements of Discrete Mathematics As a Sequence of Problems (in Russian),
MCCME, Moscow, 2016. Update of a part: \url{http://www.mccme.ru/circles/oim/discrbook.pdf}}

\newcommand{\gdikrs}{\bibitem[GDI]{GDI} * {\it А.А. Глибичук, А.Б. Дайняк, Д.Г. Ильинский, А.Б. Купавский, А.М. Райгородский, А.Б. Скопенков, А.А. Чернов,} Элементы дискретной математики в задачах, М, МЦНМО, 2016.
Обновляемая версия части книги: \url{http://www.mccme.ru/circles/oim/discrbook.pdf}}

\newcommand{\giso}{\bibitem[Gi71]{Gi71} * {\it S. Gitler,} Immersion and Embedding of Manifolds,
Proc. Symp. Pure Math. 22, 87-96 (1971).}

\newcommand{\gkp}{\bibitem[GKP]{GKP} * {\it R. Graham, D. Knuth, and O. Patashnik,} Concrete Mathematics: A Foundation for Computer Science, Addison–Wesley, first published in 1989, \url{https://www.csie.ntu.edu.tw/~r97002/temp/Concrete\%20Mathematics\%202e.pdf}.}

\newcommand{\gmpptw}{\bibitem[GMP+]{GMP+} \emph{X. Goaoc, I. Mabillard, P. Pat\'ak, Z. Pat\'akov\'a, M. Tancer, U. Wagner}, On Generalized Heawood Inequalities for Manifolds: a van Kampen--Flores-type Nonembeddability Result,
Israel J. Math., 222(2) (2017) 841-866. arXiv:1610.09063.}


\newcommand{\gppot}{\bibitem[GPP+]{GPP+} \emph{X. Goaoc, P. Pat\'ak, Z. Pat\'akov\'a, M. Tancer, and U. Wagner.} Bounding Helly numbers via Betti numbers. In 31st International Symposium on Computational Geometry, volume 34
of LIPIcs. Leibniz Int. Proc. Inform., pp. 507-521. Schloss Dagstuhl. Leibniz-Zent. Inform., Wadern, 2015. Full version: arXiv:1310.4613.}

\newcommand{\group}{\bibitem[Gr]{Gr} * \url{https://en.wikipedia.org/wiki/Groupthink}}

\newcommand{\grsz}{\bibitem[Gr69]{Gr69} \emph{B. Gr\"unbaum.} Imbeddings of simplicial complexes. Comment. Math. Helv., 44:1, 502--513, 1969.}


\newcommand{\gres}{\bibitem[Gr86]{Gr86} * \emph{M. Gromov}, Partial Differential Relations,
Ergebnisse der Mathematik und ihrer Grenzgebiete (3), Springer Verlag, Berlin-New York, 1986.}

\newcommand{\groz}{\bibitem[Gr10]{Gr10} \emph{M. Gromov,}
\newblock Singularities, expanders and topology of maps. Part 2: From combinatorics to topology via algebraic isoperimetry, \newblock Geometric and Functional Analysis 20 (2010), no.~2, 416--526.}

\newcommand{\grsn}{\bibitem[GR79]{GR79} \emph{J. L. Gross	and R. H. Rosen}, A linear time planarity algorithm for 2-complexes, Journal of the ACM, 26:4 (1979), 611--617.}

\newcommand{\gs}{\bibitem[GS]{GS} \emph{М. Гортинский и О. Скрябин.} Критерий вложимости графов в плоскость вдоль прямой, препринт.}

\newcommand{\gssn}{\bibitem[GS79]{GS} \emph{P.~M. Gruber and R.~Schneider,} Problems in geometric convexity. In {\em Contributions to geometry (Proc. Geom. Sympos., Siegen, 1978)}, 255--278. Birkh{\"a}user, Basel-Boston, Mass., 1979.}

\newcommand{\gsnn}{\bibitem[GS99]{GS99} \emph{R. Gompf and A. Stipsicz,}
4-manifolds and Kirby calculus, GSM20, AMS, Providence, RI, 1999.}


\newcommand{\gszs}{\bibitem[GS06]{GS06} \emph{D. Goncalves and A. Skopenkov,} Embeddings of homology equivalent manifolds with boundary, Topol. Appl., 153:12 (2006) 2026-2034. arxiv:1207.1326.}

\newcommand{\gssoe}{\bibitem[GSS+]{GSS+} * Projections of skew lines, presented by A. Gaifullin, A. Shapovalov, A. Skopenkov and M. Skopenkov, \url{http://www.turgor.ru/lktg/2001/index.php}.}

\newcommand{\gtes}{\bibitem[GT87]{GT87} * \emph{J. L. Gross and T. W. Tucker.}
Topological graph theory. New York: Wiley-Interscience, 1987.}

\newcommand{\guzn}{\bibitem[Gu09]{Gu09} \emph{A. Gundert.} On the complexity of embeddable simplicial complexes. Diplomarbeit, Freie Universit\"at Berlin, 2009. 	arXiv:1812.08447.}


\newcommand{\ha}{\bibitem[Ha]{Ha} * \emph{F. Harary.} Graph theory.
Рус. пер.: Ф. Харари. Теория графов. М., Мир, 1973.}

\newcommand{\hats}{\bibitem[Ha37]{Ha37} \emph{W. Hantzsche,} Einlagerung von Mannigfaltigkeiten in euklidische R\" aume, Math. Zeitschrift, 43:1 (1937) 38--58.}

\newcommand{\hastk}{\bibitem[Ha62k]{Ha62k} {\em A.~Haefliger,}  Knotted $(4k-1)$-spheres in $6k$-space, Ann. of Math. 75 (1962) 452--466.}

\newcommand{\hastl}{\bibitem[Ha62l]{Ha62l} \emph{A. Haefliger,} Differentiable links, Topology, 1 (1962) 241--244.}

\newcommand{\hast}{\bibitem[Ha63]{Ha63} \emph{A.~Haefliger,} Plongements differentiables dans le domain stable, Comment. Math. Helv. 36 (1962-63) 155--176.}

\newcommand{\hassa}{\bibitem[Ha66A]{Ha66A} \textit{A. Haefliger}. Differential embeddings of~$S^n$ in $S^{n+q}$ for $q>2$. Ann. Math. (2), 83 (1966), 402--~436.}

\newcommand{\hass}{\bibitem[Ha66C]{Ha66C} \emph{A.~Haefliger,}  Enlacements de spheres en codimension superiure \`a 2, Comment. Math. Helv. 41 (1966-67) 51--72.}

\newcommand{\hase}{\bibitem[Ha68]{Ha68} \emph{A. Haefliger,} Knotted Spheres and Related Geometric Topic,
in Proc. Int. Congr. Math., Moscow, 1966 (Mir, Moscow, 1968), 437--445.}

\newcommand{\hasn}{\bibitem[Ha69]{Ha69} \emph{L.~S.~Harris,} Intersections and embeddings of polyhedra, Topology 8 (1969) 1--26.}

\newcommand{\hasf}{\bibitem[Ha74]{Ha74} * \emph{P. Halmos,} How to talk mathematics. Notices of the Amer. Math. Soc., 21 (1974) 155--158.}

\newcommand{\haef}{\bibitem[Ha84]{Ha84} \emph{N. Habegger,} Obstruction to embedding disks II: a proof of a conjecture by Hudson, Topol. Appl. 17 (1984).}

\newcommand{\haes}{\bibitem[Ha86]{Ha86} \emph{N. Habegger,} Knots and links in codimension greater than 2, Topology, 25:3 (1986) 253--260.}

\newcommand{\hogr}{\bibitem[HG]{HG} * \url{http://www.map.mpim-bonn.mpg.de/Homology_groups_(simplicial;_simple_definition)}}

\newcommand{\hifn}{\bibitem[Hi59]{Hi59} \emph{M. W. Hirsch.} Immersions of manifolds, Trans. Amer. Math. Soc. 93 (1959) 242--276.}

\newcommand{\hjsf}{\bibitem[HJ64]{HJ64} \emph{R. Halin and H. A. Jung.}
Karakterisierung der Komplexe der Ebene und der 2-Sph\"are, Arch. Math. 1964. 15. P.~466-469.}

\newcommand{\hkns}{\bibitem[HK96]{HK96} \emph{N. Habegger and U. Kaiser,} Homotopy classes of 2 disjoint $2p$-spheres in $\R^{3p+1}$, Topol. Appl. 71 (1996) 1--8.}

\newcommand{\hkne}{\bibitem[HK98]{HK98} \emph{N. Habegger and U. Kaiser,} Link homotopy in 2--metastable range, Topology 37:1 (1998) 75--94.}

\newcommand{\hmsnt}{\bibitem[HMS]{HMS93} * \emph{C. Hog-Angeloni, W. Metzler and A. J. Sieradski.}
Two-dimensional homotopy and combinatorial group theory. Cambridge: Cambridge Univ. Press, 1993. (London Math. Soc. Lecture Notes, 197).}

\newcommand{\ho}{\bibitem[Ho]{Ho} * The Hopf fibration, \url{https://www.youtube.com/watch?v=AKotMPGFJYk}}

\newcommand{\hozs}{\bibitem[Ho06]{Ho06} \emph{H. van der Holst,} Graphs and obstructions in four dimensions, J. Combin. Theory Ser. B 96:3 (2006), 388--404.}


\newcommand{\hpzn}{\bibitem[HP09]{HP09} \emph{H. van der Holst and R. Pendavingh,} On a graph property generalizing planarity and flatness, Combinatorica, 29 (2009) 337--361.}

\newcommand{\hssf}{\bibitem[HS64]{HS64} \emph{A. Haefliger and B. Steer,} Symmetry of linking coefficients, Comment. Math. Helv. 39 (1964) 259-270.}

\newcommand{\htsf}{\bibitem[HT74]{HT74} \emph{J. Hopcroft and R. E. Tarjan,} Efficient planarity testing, J. of the Association for Computing Machinery, 21:4 (1974) 549--568.}

\newcommand{\hufn}{\bibitem[Hu59]{hu59} * \emph{S. T. Hu,} Homotopy Theory, Academic Press, New York, 1959.}

\newcommand{\huss}{\bibitem[Hu66]{Hu66} * \emph{J.~F.~P.~Hudson,} Extending piecewise linear isotopies, Proc. London Math. Soc. (3) 16 (1966) 651--668.}

\newcommand{\husn}{\bibitem[Hu69]{Hu69} * \emph{J. F. P. Hudson.} Piecewise linear topology, W. A. Benjamin, Inc., New York-Amsterdam, 1969.}


\newcommand{\io}{\bibitem[Io]{Io} * \url{https://en.wikipedia.org/wiki/Category:Impossible_objects}}

\newcommand{\info}{\bibitem[IF]{IF} * \url{http://www.map.mpim-bonn.mpg.de/Intersection_form}}

\newcommand{\irsf}{\bibitem[Ir65]{Ir65} \emph{M.~C.~Irwin,} Embeddings of polyhedral manifolds, Ann. of Math. (2)
82 (1965) 1--14.}

\newcommand{\isot}{\bibitem[Is]{Is} * \url{http://www.map.mpim-bonn.mpg.de/Isotopy}}


\newcommand{\jqnt}{\bibitem[JQ93]{JQ93} * \emph{A. Jaffe, F. Quinn,} ``Theoretical mathematics'': Toward a cultural synthesis of mathematics and theoretical physics. Bull.Am.Math.Soc. 29 (1993) 1-13. arXiv:math/9307227.}

\newcommand{\jozt}{\bibitem[Jo02]{Jo02} \emph{C. M. Johnson.} An obstruction to embedding a simplicial $n$-complex into a $2n$-manifold, Topology Appl. 122:3 (2002) 581--591.}

\newcommand{\jvz}{\bibitem[JVZ]{JVZ} D. Joji\'c, S. T. Vre\'cica, R. T. \v Zivaljevi\' c,
Topology and combinatorics of 'unavoidable complexes', arXiv:1603.08472v1.}


\newcommand{\kalai}{\bibitem[Ka]{Ka} \emph{G. Kalai,} From Oberwolfach: The Topological Tverberg Conjecture is False, `Combinatorics and more' blog post, February 6, 2015, \url{gilkalai.wordpress.com}}

\newcommand{\kano}{\bibitem[Ka91]{Ka91} \emph{G. Kalai,} The diameter of graphs of convex polytopes and $f$-vector theory, Applied geometry and discrete mathematics, DIMACS Ser. Discrete Math. Theoret. Comput. Sci., vol. 4, Amer. Math. Soc., Providence, RI, 1991, pp. 387--411.}

\newcommand{\kefn}{\bibitem[Ke59]{Ke59} \emph{M. Kervaire.} An interpretation of G. Whitehead's generalization of H. Hopf's invariant, Ann. of Math. 62 (1959), 345--362.}

\newcommand{\kh}{\bibitem[Kh]{Kh} \emph{А.И. Храбров.} Руководство по чтению лекций
\url{http://vm.tstu.tver.ru/topics/pdf_tests/lection.pdf}}

\newcommand{\kho}{\bibitem[Kho]{Kho} \emph{N. Khoroshavkina.} A simple characterization of graphs of cutwidth 2, arXiv:1811.06716.}

\newcommand{\kkrot}{\bibitem[KKR]{KKR} \emph{K. Kawarabayashi, Y. Kobayashi and B. Reed.} The disjoint paths problem in quadratic time, J. of Comb. Theory, Ser. B, 102:2 (2012), 424--435.}

\newcommand{\kln}{\bibitem[KLN]{KLN} \emph{ J. Kratochv\'il, A. Lubiw and J. Ne\v set\v ril.} Noncrossing subgraphs in topological layouts, SIAM J. on Discr. Math. 4(2) (1991), 223--244.}

\newcommand{\kmsth}{\bibitem[KM63]{KM63} \emph{M. A. Kervaire and J. W. Milnor,} Groups of homotopy spheres. I,  Ann. of Math. (2) 77 (1963), 504-537.}

\newcommand{\kozeru}{\bibitem[Ko18]{Ko18} * \emph{Е. Колпаков.}
Доказательство теоремы Радона при помощи понижения размерности, Мат. Просвещение, 23 (2018), arXiv:1903.11055.}

\newcommand{\koze}{\bibitem[Ko18]{Ko18} * \emph{E. Kolpakov.}
A proof of Radon Theorem via lowering of dimension, Mat. Prosveschenie, 23 (2018), arXiv:1903.11055.}

\newcommand{\ko}{\bibitem[Ko]{Ko} \emph{E. Kolpakov.} A `converse' to the Constraint Lemma, arXiv:1903.08910.}

\newcommand{\koon}{\bibitem[Ko19]{Ko19} \emph{E. Kogan.} Linking of three triangles in 3-space, arXiv:1908.03865.}

\newcommand{\koto}{\bibitem[Ko21]{Ko21} \emph{E. Kogan.} On the rank of $\Z_2$-matrices with free entries on the diagonal, arXiv:2104.10668.}

\newcommand{\koee}{\bibitem[Ko88]{Ko88} \emph{U. Koschorke.} Link maps and the geometry of their invariants,
Manuscripta Math. 61:4 (1988) 383--415.}

\newcommand{\kono}{\bibitem[Ko91]{Ko91} \emph{U. Koschorke.} Link homotopy with many components,
Topology 30:2 (1991) 267--281.}

\newcommand{\kons}{\bibitem[Ko97]{Ko97} \emph{U. Koschorke.} A generalization of Milnor's $\mu$-invariants to higher-dimensional link maps, Topology 36:2 (1997) 301--324.}

\newcommand{\kps}{\bibitem[KPS]{KPS} * \emph{A. Kaibkhanov, D. Permyakov and A. Skopenkov.}
Realization of graphs with rotation, \url{http://www.turgor.ru/lktg/2005/3/index.htm}.}

\newcommand{\krzz}{\bibitem[Kr00]{Kr00} \emph{V. S. Krushkal.} Embedding obstructions and 4-dimensional thickenings of 2-complexes, Proc. Amer. Math. Soc. 128:12 (2000) 3683--3691. arXiv:math/0004058. }

\newcommand{\ksnn}{\bibitem[KS99]{KS99} * \emph{П. Кожевников и А. Скопенков.} Узкие деревья на плоскости, Мат. Образование. 1999. 2-3. С.~126-131.}

\newcommand{\kstz}{\bibitem[KS20]{KS20} \emph{R. Karasev and A. Skopenkov.}
Some `converses' to intrinsic linking theorems, Discr. Comp. Geom., 70:3 (2023), 921--930, arXiv:2008.02523.}


\newcommand{\ksto}{\bibitem[KS21]{KS21} * \emph{E. Kogan and A. Skopenkov.} A short exposition of the Patak-Tancer theorem on non-embeddability of $k$-complexes in $2k$-manifolds,  arXiv:2106.14010.}

\newcommand{\kstoe}{\bibitem[KS21e]{KS21e} \emph{E. Kogan and A. Skopenkov.}
Embeddings of $k$-complexes in $2k$-manifolds and minimum rank of partial symmetric matrices, arXiv:2112.06636v2.}

\newcommand{\kstf}{\bibitem[KS24]{KS24} \emph{R. Karasev and A. Skopenkov.}
Short proofs of Tverberg-type theorems for cell complexes, Discr. Comp. Geom., (2025), arXiv:2405.05629.}


\newcommand{\kutt}{\bibitem[Ku23]{Ku23} \emph{W. K\"uhnel.} Generalized Heawood Numbers, The Electronic Journal of Combinatorics, 30:4 (2023) \#P4.17.}


\newcommand{\kuse}{\bibitem[Ku68]{Ku68} * \emph{К. Куратовский.} Топология. Т.~1,~2. М.: Мир, 1969.}

\newcommand{\kunfo}{\bibitem[Ku94]{Ku94} \emph{W. K\"uhnel.} Manifolds in the skeletons of convex polytopes, tightness, and generalized Heawood inequalities. In Polytopes: abstract, convex and computational (Scarborough, ON, 1993), volume 440 of NATO Adv. Sci. Inst. Ser. C Math. Phys. Sci., pp. 241--247. Kluwer
Acad. Publ., Dordrecht, 1994.}


\newcommand{\kunf}{\bibitem[Ku95]{Ku95} * \emph{W. K\"uhnel}, Tight Polyhedral Submanifolds and Tight Triangulations, Lecture Notes in Math. 1612, Springer, 1995.}

\newcommand{\kytz}{\bibitem[Ky20]{Ky20} \emph{J. Kyn{\v{c}}l,} Simple Realizability of Complete Abstract Topological Graphs Simplified, Discr. Comp. Geom., 64 (2020) 1--27, arXiv:1608.05867.}


\newcommand{\lazz}{\bibitem[La00]{La00} \emph{F. Lasheras.} An obstruction to 3-dimensional thickening,
Proc. Amer. Math. Soc. 2000. 128. P.~893-902.}

\newcommand{\lfma}{\bibitem[LF]{LF} \url{http://www.map.mpim-bonn.mpg.de/Linking_form}}

\newcommand{\lloe}{\bibitem[LL18]{LL18} \emph{A.S. Levine and T. Lidman.} Simply connected, spineless 4-manifolds, Forum of Math., Sigma, 7 (2019) e14, 1--11, arxiv:1803.01765.}

\newcommand{\lo}{\bibitem[Lo]{Lo} M.~de~Longueville. Notes on the topological Tverberg theorem.
Discrete Math.  247 (2002), no.~1--3, 271--297.
(The paper first appeared in
Discrete Math. 241 (2001) 207--233, but the original version suffered from serious publisher's typesetting errors.)}

\newcommand{\loot}{\bibitem[Lo13]{Lo13} \emph{M. de Longueville.} A course in topological combinatorics. Universitext. Springer, New York (2013).}

\newcommand{\lssn}{\bibitem[LS69]{LS69} \emph{W. B. R. Lickorish and L. C. Siebenmann.}
Regular neighborhoods and the stable range,  Trans. Amer. Math. Soc.. 1969. 139. P.~207-230.}

\newcommand{\lsne}{\bibitem[LS98]{LS98} \emph{L. Lovasz and A. Schrijver,}
A Borsuk theorem for antipodal links and a spectral characterization of linklessly embeddable graphs, Proc. Amer. Math. Soc. 126:5 (1998), 1275-1285.}

\newcommand{\ltof}{\bibitem[LT14]{LT14} \emph{E. Lindenstrauss and M. Tsukamoto,} Mean dimension and an embedding problem: an example, Israel J. Math. 199 (2014).}


\newcommand{\lyzf}{\bibitem[LY04]{LY04} * \emph{Y. Lin and A. Yang,} On 3-cutwidth critical graphs, Discrete Mathematics, 275 (2004), 339--346.}

\newcommand{\lz}{\bibitem[LZ]{LZ} * \emph{S. Lando and A. Zvonkin.} Embedded Graphs. Springer.}


\newcommand{\maez}{\bibitem[Ma80]{Ma80} * R. Mandelbaum, {\em Four-Dimensional Topology: An introduction},
Bull. Amer. Math. Soc. (N.S.) 2 (1980) 1-159.}

\newcommand{\mast}{\bibitem[Ma73]{Ma73} \emph{С. В. Матвеев.} Специальные остовы кусочно-линейных многообразий, Мат. Сборник. 1973. 92. С.~282-293.}

\newcommand{\maste}{\bibitem[Ma73]{Ma73} \emph{S. V. Matveev.} Special skeletons of PL manifolds (in Russian), Mat. Sbornik. 1973. 92. P.~282-293.}

\newcommand{\manz}{\bibitem[Ma90]{Ma90} \emph{W. S.  Massey.} Homotopy classification of 3-component links of codimension greater than 2, Topol.  Appl. 34 (1990) 269--300.}

\newcommand{\mans}{\bibitem[Ma97]{Ma97} \emph{Yu. Makarychev.} A short proof of Kuratowski's graph planarity criterion, J. of Graph Theory, 25 (1997), 129--131.}

\newcommand{\matns}{\bibitem[Mat97]{Mat97} \emph{J. Matou\v sek.} A Helly-type theorem for unions of convex sets. Discr. Comp. Geom., 18:1 (1997) 1-12.}

\newcommand{\mazt}{\bibitem[Ma03]{Ma03} * \emph{J.~Matou{\v{s}}ek.} Using the {B}orsuk-{U}lam theorem:
Lectures on topological methods in combinatorics and geometry. Springer Verlag, 2008.}


\newcommand{\mazf}{\bibitem[Ma05]{Ma05} \emph{V. Manturov.} A proof of the Vasiliev conjecture on the planarity of singular links, Izv. RAN 2005.}

\newcommand{\metn}{\bibitem[Me29]{Me29} \emph{K. Menger.} \"Uber pl\"attbare Dreiergraphen und Potenzen nicht pl\"attbarer Graphen, Ergebnisse Math. Kolloq., 2 (1929) 30--31.}

\newcommand{\mezf}{\bibitem[Me04]{Me04} \emph{S. Melikhov.} Sphere eversions and realization of mappings, Trudy MIAN 247 (2004) 159-181 (in Russian) arXiv:math.GT/0305158.}

\newcommand{\mezs}{\bibitem[Me06]{Me06} \emph{S. A. Melikhov}, The van Kampen obstruction and its relatives, 	
Proc. Steklov Inst. Math 266 (2009), 142-176 (= Trudy MIAN 266 (2009), 149-183), arXiv:math/0612082.}

\newcommand{\meoo}{\bibitem[Me11]{Me11} \emph{S. A. Melikhov}, Combinatorics of embeddings, arXiv:1103.5457.}

\newcommand{\meos}{\bibitem[Me17]{Me17} \emph{S. Melikhov,} Gauss type formulas for link map invariants, arXiv:1711.03530.}

\newcommand{\meoe}{\bibitem[Me18]{Me18} \emph{S. A. Melikhov,} A triple-point Whitney trick, J. Topol. Anal., 2018, 1--6. arXiv:2210.04016.}


\newcommand{\metz}{\bibitem[Me20]{Me20} \emph{S. A. Melikhov,} Topological isotopy and Cochran's derived invariants, in `Topology, Geometry, and Dynamics: Rokhlin Memorial', Contemporary Mathematics, 772, AMS, Providence, RI, 2021. arXiv:2011.01409.}

\newcommand{\mett}{\bibitem[Me22]{Me22} \emph{S. A. Melikhov,} Embeddability of joins and products of polyhedra, Topol. Methods in Nonlinear Analysis, 60:1 (2022), 185-201. arXiv:2210.04015.}

\newcommand{\miff}{\bibitem[Mi54]{Mi54} \emph{J. Milnor,} Link groups, Ann. of Math. 59 (1954), 177--195.}

\newcommand{\miso}{\bibitem[Mi61]{Mi61} \emph{J. Milnor,} A procedure for killing homotopy groups of differentiable manifolds, Proc. Sympos. Pure Math, Vol. III (1961), 39--55.}

\newcommand{\mins}{\bibitem[Mi97]{Mi97} \emph{P. Minc.} Embedding simplicial arcs into the plane, Topol. Proc. 1997. 22. 305--340.}


\newcommand{\adnsvr}{\bibitem[MNS]{MNS} * \emph{А. Мирошников, О. Никитенко и А. Скопенков.} Циклы в графах и в гиперграфах: в направлении теории гомологий, Мат. Просвещение, 35 (2025), 137-184, arXiv:2406.16705.}

\newcommand{\dmnse}{\bibitem[MNS]{MNS} * \emph{A. Miroshnikov, O. Nikitenko, A. Skopenkov.}
Cycles in graphs and in hypergraphs: towards homology theory (in Russian), Mat. Prosveschenie, 35 (2025), 137-184, arXiv:2406.16705.}

\newcommand{\moss}{\bibitem[Mo77]{Mo77} * \emph{E. E. Moise.} Geometric Topology in Dimensions 2 and 3 (GTM), Springer-Verlag, 1977.}

\newcommand{\moen}{\bibitem[Mo89]{Mo89} \textit{B. Mohar}. An obstruction to embedding graphs in
surfaces. Discrete Math. 78 (1989) 135--142.}

\newcommand{\moze}{\bibitem[Mo08]{Mo08} \textit{T. Moriyama}. An invariant of embeddings of 3-manifolds in 6-manifolds and Milnor's triple linking number, J. Math. Sci. Univ. Tokyo, 18 (2011), 193--237. arXiv:0806.3733.}


\newcommand{\mrst}{\bibitem[MRS+]{MRS+} \emph{A. de Mesmay, Y. Rieck, E. Sedgwick, M. Tancer,}
Embeddability in $\R^3$ is NP-hard. arXiv:1708.07734.}

\newcommand{\mesczs}{\bibitem[MS06]{MS06} \emph{S.A. Melikhov, E.V. Shchepin,} The telescope approach to embeddability of compacta. arXiv:math.GT/0612085.}

\newcommand{\msos}{\bibitem[MS17]{MS17}  \emph{T. Maciazek, A. Sawicki.} Homology groups for particles on one-connected graphs
J. Math. Phys. 58, 062103 (2017). arXiv:1606.03414.}

\newcommand{\mstwof}{\bibitem[MST+]{MST+} \emph{J. Matou\v sek, E. Sedgwick, M. Tancer, U. Wagner}, Embeddability in the 3-sphere is decidable, Journal of the ACM 65:1 (2018) 1--49, arXiv:1402.0815.}


\newcommand{\mtzo}{\bibitem[MT01]{MT01} * \emph{B. Mohar and C. Thomassen.} Graphs on Surfaces.
The John Hopkins University Press, 2001.}

\newcommand{\mtwoz}{\bibitem[MTW10]{MTW10} \emph{J. Matou\v sek, M. Tancer, U. Wagner.} A geometric proof of
the colored Tverberg theorem, Discr. and Comp. Geometry, 47:2 (2012), 245--265. arXiv:1008.5275.}


\newcommand{\mtwoo}{\bibitem[MTW]{MTW} \emph{J. Matou\v sek, M. Tancer, U. Wagner.}
Hardness of embedding simplicial complexes in $\R^d$, J. Eur. Math. Soc. 13:2 (2011), 259--295. arXiv:0807.0336.}


\newcommand{\musf}{\bibitem[Mu74]{Mu74} \emph{J. Muncres,} Elementary differential topology, Complement to the book by J. W. Milnor and J. D. Stasheff, {\it Characteristic Classes}, Ann. of Math. St. 76 (1974), Princeton Univ. Press, Princeton, NJ.}

\newcommand{\mwoe}{\bibitem[MW18]{MW18} * \emph{F. Manin, S. Weinberger.} Algorithmic aspects of immersibility and embeddability, Intern. Math. Res. Notices, rnae170. arXiv:1812.09413.}


\newcommand{\mwsn}{\bibitem[MW69]{MW69} * \emph{J. MacWilliams}. Orthogonal matrices over finite fields. Amer. Math. Monthly, 76 (1969) 152--164.}

\newcommand{\mwofo}{\bibitem[MW14]{MW14} \emph{I. Mabillard and U. Wagner.} Eliminating Tverberg Points, I. An Analogue of the Whitney Trick, Proc. of the 30th Annual Symp. on Comp. Geom. (SoCG'14), ACM, New York, 2014, pp. 171--180.}

\newcommand{\mwof}{\bibitem[MW15]{MW15} \emph{I. Mabillard and U. Wagner.}
Eliminating Higher-Multiplicity Intersections, I. A Whitney Trick for Tverberg-Type Problems. arXiv:1508.02349.}


\newcommand{\mwos}{\bibitem[MW16]{MW16} \emph{I. Mabillard and U. Wagner.} Eliminating Higher-Multiplicity Intersections, II. The Deleted Product Criterion in the $r$-Metastable Range. arXiv:1601.00876v2.}

\newcommand{\mwosd}{\bibitem[MW16']{MW16'} \emph{I. Mabillard and U. Wagner.} Eliminating Higher-Multiplicity Intersections, II. The Deleted Product Criterion in the r-Metastable Range,
Proceedings of the 32nd Annual Symposium on Computational Geometry (SoCG'16).}


\newcommand{\neno}{\bibitem[Ne91]{Ne91} \emph{S. Negami.} Ramsey theorems for knots, links and spatial graphs,
Trans. Amer. Math. Soc., 324 (1991), 527--541.}



\newcommand{\nizz}{\bibitem[Ni00]{Ni00} \emph{R. Nikkuni.} The second skew-symmetric cohomology group and spatial embeddings of graphs, J. Knot Theory Ram. 9 (2000), 387--411.}

\newcommand{\nkon}{\bibitem[NKS]{NKS} * \emph{L. T. Nguyen, J. Kim, B. Shim.}
Low-Rank Matrix Completion: A Contemporary Survey. arXiv:1907.11705.}

\newcommand{\noss}{\bibitem[No76]{No76} * \emph{С. П. Новиков.} Топология-1. М.: Наука, 1976. (Итоги науки и техники. ВИНИТИ. Современные проблемы математики. Основные направления, 12).}

\newcommand{\nszn}{\bibitem[NS09]{NS09} \emph{I. Novik and E. Swartz,} Socles of Buchsbaum modules, complexes and posets, Adv. Math. 222 (2009), 2059-2084. arXiv:0711.0783.}

\newcommand{\nwns}{\bibitem[NW97]{NW97} \emph{A. Nabutovsky, S. Weinberger}. Algorithmic aspects of homeomorphism problems. arXiv:math/9707232.}


\newcommand{\omoe}{\bibitem[Om18]{Om18} * \emph{А. Омельченко,} Теория графов. М.: МЦНМО, 2018.}

\newcommand{\orszo}{\bibitem[ORS]{ORS} \emph{A. Onischenko, D. Repov\v s and A. Skopenkov.}
Resolutions of 2-polyhedra by fake surfaces and embeddings into $\R^4$, Contemp. Math.  288 (2001) 396--400.}

\newcommand{\ossf}{\bibitem[OS74]{OS74} \emph{R. P. Osborne and R. S. Stevens.} Group presentations
corresponding to spines of 3-manifolds, I, Amer. J.~Math. 1974. 96. P.~454-471; II, Amer. J.~Math. 1977. 234.
P.~213-243; III, Amer. J.~Math. 1977. 234 P.~245-251.}


\newcommand{\oz}{\bibitem[Oz]{Oz} \emph{M. \"Ozaydin,} Equivariant maps for the symmetric group, unpublished,
\url{http://minds.wisconsin.edu/handle/1793/63829}.}

\newcommand{\panof}{\bibitem[Pan15]{Pan15} \emph{K. Panagiotis.} A note on the topology of irreducible $SO(3)$-manifolds, 	arXiv:1508.06150.}

\newcommand{\paof}{\bibitem[Pa15]{Pa15} \emph{S. Parsa,} On links of vertices in simplicial $d$-complexes embeddable in the Euclidean $2d$-space, Discrete Comput. Geom. 59:3 (2018), 663--679.
This is arXiv:1512.05164v4 up to numbering of sections, theorems etc.; we refer to numbering in arxiv version.
Correction: Discrete Comput. Geom. 64:3 (2020) 227--228.}

\newcommand{\paoe}{\bibitem[Pa18]{Pa18} \emph{S. Parsa,} On links of vertices in simplicial $d$-complexes
embeddable in the euclidean $2d$-space, arXiv:1512.05164v6.}

\newcommand{\patz}{\bibitem[Pa20]{Pa20} \emph{S. Parsa,} On links of vertices in simplicial $d$-complexes
embeddable in the euclidean $2d$-space, arXiv:1512.05164v8.}


\newcommand{\patzl}{\bibitem[Pa20]{Pa20} \emph{S. Parsa,}
Correction to: On the Links of Vertices in Simplicial $d$-Complexes Embeddable in the Euclidean $2d$-Space,
Discrete Comput. Geom. 64:3 (2020) 227--228.}

\newcommand{\patza}{\bibitem[Pa20]{Pa20} \emph{S. Parsa,} On the Smith classes, the van Kampen obstruction and embeddability of $[3]*K$, arXiv:2001.06478.}

\newcommand{\patzb}{\bibitem[Pa20b]{Pa20b} \emph{S. Parsa,} On the embeddability of $[3]*K$, arXiv:2001.06506.}

\newcommand{\pato}{\bibitem[Pa21]{Pa21} \emph{S. Parsa,} Instability of the Smith index under joins and applications to embeddability, Trans. Amer. Math. Soc. 375 (2022), 7149--7185, arXiv:2103.02563.}

\newcommand{\pak}{\bibitem[Pa]{Pa} * \emph{I. Pak}, Lectures on Discrete and Polyhedral Geometry, \url{http://www.math.ucla.edu/~pak/geompol8.pdf}.}

\newcommand{\peze}{\bibitem[Pe08]{Pe08} \emph{Д. Пермяков.} Классификация погружений графов в плоскость,
Вестник МГУ, сер.1, 2008, N5, 55-56.}

\newcommand{\peos}{\bibitem[Pe16]{Pe16} \emph{Д. Пермяков.} Матем. сб., 207:6 (2016),  93--112.}

\newcommand{\pest}{\bibitem[Pe72]{Pe72} * \emph{B. B. Peterson.} The Geometry of Radon's Theorem, Amer. Math. Monthly 79 (1972), 949-963.}


\newcommand{\prnf}{\bibitem[Pr95]{Pr95} * \emph{V. V. Prasolov.} Intuitive topology. Amer. Math. Soc., Providence, R.I., 1995.}

\newcommand{\prnfr}{\bibitem[Pr95]{Pr95} * \emph{В. В. Прасолов.} Наглядная топология. М.: МЦНМО, 1995.}


\newcommand{\przs}{\bibitem[Pr06]{Pr06} * \emph{V. V. Prasolov.}
Elements of Combinatorial and Differential Topology, 2006, GSM 74, Amer. Math. Soc., Providence, RI.}

\newcommand{\przsru}{\bibitem[Pr04]{Pr04} * \emph{В. В. Прасолов.}
Элементы комбинаторной и дифференциальной топологии. М.: МЦНМО, 2004. \url{http://www.mccme.ru/prasolov}.}

\newcommand{\przse}{\bibitem[Pr07]{Pr07} * \emph{V. V. Prasolov.} Elements of homology theory. 2007, GSM 74, Amer. Math. Soc., Providence, RI.}


\newcommand{\przseru}{\bibitem[Pr06]{Pr06} * \emph{В. В. Прасолов.} Элементы теории гомологий. М.: МЦНМО, 2006.}


\newcommand{\psns}{\bibitem[PS96]{PS96} * \emph{V. V. Prasolov, A. B. Sossinsky } Knots, Links, Braids, and 3-manifolds. Amer. Math. Soc. Publ., Providence, R.I., 1996.}


\newcommand{\pszf}{\bibitem[PS05]{PS05} * \emph{В. В. Прасолов и М. Б. Скопенков.}
Рамсеевская теория зацеплений, Мат. Просвещение. 2005. 9. С.~108--115.}

\newcommand{\pszfen}{\bibitem[PS05]{PS05} * \emph{V. V. Prasolov and M.B. Skopenkov.}
Ramsey link theory, Mat, Prosvescheniye, 9 (2005), 108--115.}

\newcommand{\psoo}{\bibitem[PS11]{PS11} \emph{Y. Ponty and C. Saule.} A combinatorial framework for designing (pseudoknotted) RNA algorithms, Proc. of the 11th Intern. Workshop on Algorithms in Bioinformatics, WABI'11, 250--269.}


\newcommand{\pstz}{\bibitem[PS20]{PS20} \emph{S. Parsa and A. Skopenkov.} On embeddability of joins and their `factors', Topol. Appl., 326 (2023) 108409, arXiv:2003.12285.}



\newcommand{\psszn}{\bibitem[PSS]{PSS} \emph{M. J. Pelsmajer, M. Schaefer and D. Stasi.} Strong Hanani-Tutte on the projective plane. SIAM J. Discrete Math., 23:3 (2009) 1317--1323.}

\newcommand{\psszs}{\bibitem[PSS]{PSS} \emph{M. J. Pelsmajer, M. Schaefer, and D. \v Stefankovi\v c.}
Removing even crossings. J. Combin. Theory Ser. B, 97(4):489–500, 2007.}

\newcommand{\pton}{\bibitem[PT19]{PT19} \emph{P. Pat\'ak and M. Tancer.} Embeddings of $k$-complexes into $2k$-manifolds. Discrete Comput. Geom. 71 (2024), 960--991. arXiv:1904.02404v4.}

\newcommand{\pw}{\bibitem[PW]{PW} \emph{I. Pak, S. Wilson}, G\lowercase{EOMETRIC REALIZATIONS OF POLYHEDRAL COMPLEXES}, \linebreak \url{http://www.math.ucla.edu/~pak/papers/Fary-full31.pdf}.}


\newcommand{\razf}{\bibitem[RA05]{RA05} * \emph{J. L. Ram\'irez Alfons\'in.} Knots and links in spatial graphs: a survey. Discrete Math., 302 (2005), 225--242.}

\newcommand{\rep}{\bibitem[Rep]{Rep} Referee's report on the paper ``Some `converses' to intrinsic linking theorems', \url{https://www.mccme.ru/circles/oim/materials/ksreport.pdf}}

\newcommand{\rnoo}{\bibitem[RN11]{RN11} * \emph{R. L. Ricca, B. Nipoti.} Gauss' linking number revisited.
J. of Knot Theory and Its Ramif. 20:10 (2011) 1325--1343. \url{https://www.maths.ed.ac.uk/~v1ranick/papers/ricca.pdf} .}

\newcommand{\rrstz}{\bibitem[RRS]{RRS} * \emph{V. Retinskiy, A. Ryabichev and A. Skopenkov.}
Motivated exposition of the proof of the Tverberg Theorem (in Russian).
Mat. Prosveschenie, 27 (2021), 166--169. arXiv:2008.08361.}


\newcommand{\rssec}{\bibitem[RS68]{RS68} \emph{C. P. Rourke and B. J. Sanderson,} Block bundles II, Ann. of Math. (2), 87 (1968) 431--483.}

\newcommand{\rsst}{\bibitem[RS72]{RS72} * \emph{C. P. Rourke and B. J. Sanderson,}
\newblock Introduction to Piecewise-Linear Topology,
\newblock \emph{Ergebn.\ der Math.} 69, Springer-Verlag, Berlin, 1972.}

\newcommand{\rsstr}{\bibitem[RS72]{RS72} * \emph{К. П. Рурк и Б. Дж. Сандерсон.} Введение в кусочно-линейную топологию, Москва. Мир. 1974.}

\newcommand{\rsns}{\bibitem[RS96]{RS96} * \emph{D. Repov\v s and A. B. Skopenkov.}
Embeddability and isotopy of polyhedra in Euclidean spaces,
Proc. of the Steklov Inst. Math. 1996. 212. P.~173-188.}

\newcommand{\rsne}{\bibitem[RS98]{RS98} \emph{D. Repov\v s and A. B. Skopenkov.}
A deleted product criterion for approximability of a map by embeddings, Topol. Appl. 1998. 87 P.~1-19.}

\newcommand{\rsnn}{\bibitem[RS99]{RS99} * \emph{D. Repov\v s and A. B. Skopenkov.} New results on embeddings of polyhedra and manifolds into Euclidean spaces,
Russ. Math. Surv. 54:6 (1999), 1149--1196.}


\newcommand{\rsnnd}{\bibitem[RS99']{RS99'} * \emph{Д. Реповш и А. Скопенков.}
Кольца Борромео и препятствия к вложимости, Труды МИРАН. 1999. 225. С.~331-338.}

\newcommand{\rszz}{\bibitem[RS00]{RS00} \emph{D. Repov\v s and A. Skopenkov.} Cell-like resolutions of polyhedra by special ones,  Colloq. Math. 2000. 86:2. P. 231--237.}

\newcommand{\rszzd}{\bibitem[RS00']{RS00'} * \emph{Д. Реповш и А. Скопенков.} Характеристические классы для начинающих, Мат. Просвещение. 2000. 4. С.~151-176.}

\newcommand{\rszo}{\bibitem[RS01]{RS01} \emph{D. Repovs and A. Skopenkov.} On contractible $n$-dimensional compacta, non-embeddable into $\R^{2n}$, Proc. Amer. Math. Soc. 129 (2001) 627--628.}

\newcommand{\rszt}{\bibitem[RS02]{RS02} * \emph{Д. Реповш и А. Скопенков.} Теория препятствий для начинающих,
Мат. Просвещение. 2002. 6. C.~60-77.}

\newcommand{\rszf}{\bibitem[RS04]{RS04} \emph{N. Robertson and P. Seymour.} Graph Minors. XX. Wagner's conjecture, J. of Comb. Theory, B, 92:2 (2004) 325--357.}

\newcommand{\rssnf}{\bibitem[RSS]{RSS95} \emph{D. Repov\v s, A. B. Skopenkov  and E. V. \v S\v cepin.}
On uncountable collections of continua and their span, Colloq. Math. 1995. 69:2. P.~289-296.}

\newcommand{\rssnfd}{\bibitem[RSS']{RSS95'} \emph{D. Repov\v s, A. B. Skopenkov and E. V \v S\v cepin.}
On embeddability of $X\times I$ into Euclidean space, Houston J.~Math. 1995. 21. P.~199-204.}

\newcommand{\rssz}{\bibitem[RSS+]{RSSZ} * \emph{A. Rukhovich, A. Skopenkov, M. Skopenkov, A. Zimin},
Realizability of hypergraphs, \url{https://www.turgor.ru/lktg/2013/1/1-1en.pdf} .}


\newcommand{\rstnt}{\bibitem[RST']{RST93} \emph{N. Robertson, P. Seymour and R. Thomas}, Linkless embeddings of graphs in 3-space, Bull. of the Amer. Math. Soc., 21 (1993) 84--89.}

\newcommand{\rstno}{\bibitem[RST]{RST91} * \emph{N. Robertson, P. Seymour and R. Thomas}, A survey of
linkless embeddings, Graph Structure Theory (Seattle, WA, 1991), Contemp. Math. 147, (1993) 125--136.}


\newcommand{\rwzl}{\bibitem[RWZ+]{RWZ+} \emph{Y. Ren, C. Wen, S. Zhen, N. Lei, F. Luo, D.X. Gu},
Characteristic class of isotopy for surfaces, J. Syst. Sci. Complex. 33 (2020) 2139--2156.}


\newcommand{\saeo}{\bibitem[Sa81]{Sa81} \emph{H. Sachs.} On spatial representation of finite graphs,
in: Finite and infinite sets (Eger, 1981), 649--662, Colloq. Math. Soc. Janos Bolyai, 37, North-Holland, Amsterdam, 1984.}

\newcommand{\sano}{\bibitem[Sa91]{Sa91} \emph{K. S. Sarkaria.}
A one-dimensional Whitney trick and Kuratowski's graph planarity criterion, Israel J.~Math. 73 (1991), 79--89.}


\newcommand{\sanov}{\bibitem[Sa91g]{Sa91g} \emph{K. S. Sarkaria.} A generalized Van Kampen-Flores theorem, Proc. Amer. Math. Soc. 111 (1991), 559--565.}

\newcommand{\sant}{\bibitem[Sa92]{Sa92} \emph{K. S. Sarkaria.} Tverberg’s theorem via number fields. Israel J. Math., 79:317–320, 1992.}

\newcommand{\sann}{\bibitem[Sa99]{Sa99} O. Saeki {\em On punctured 3-manifolds in 5-sphere}, Hiroshima Math. J. 29 (1999) 255--272.}


\newcommand{\sazz}{\bibitem[Sa00]{Sa00} \emph{K. S. Sarkaria.} Tverberg partitions and Borsuk-Ulam theorems. Pacific J. Math., 196:1 (2000) 231--241.}

\newcommand{\sczf}{\bibitem[Sc04]{Sc04} \emph{T. Sch\"oneborn.} On the Topological Tverberg Theorem, arXiv:math/0405393.}


\newcommand{\scot}{\bibitem[Sc13]{Sc13} * \emph{M. Schaefer.} Hanani-Tutte and related results. In Geometry --- intuitive, discrete, and convex, Bolyai Soc. Math. Stud., 24 (2013), 259--299.
\url{http://ovid.cs.depaul.edu/documents/htsurvey.pdf} }


\newcommand{\sctz}{\bibitem[Sc20]{Sc20} \emph{M. Schaefer.} The Graph Crossing Number and
its Variants: A Survey. The Electr. J. of Comb. (2020), DS21, \url{https://www.combinatorics.org/files/Surveys/ds21/ds21v5-2020.pdf}}


\newcommand{\scef}{\bibitem[Sc84]{Sc84} \emph{E.~V.~\v S\v cepin.} Soft mappings of manifolds, Russian Math. Surveys, 39:5 (1984).}

\newcommand{\shfs}{\bibitem[Sh57]{Sh57} \emph{A. Shapiro,} Obstructions to the embedding of a complex in a Euclidean space, I, The first obstruction, Ann. Math. 66 (1957), 256--269.}


\newcommand{\shen}{\bibitem[Sh89]{Sh89} * \emph{Ю. А. Шашкин,} Неподвижные точки, М., Наука, 1989.}

\newcommand{\shoe}{\bibitem[Sh18]{Sh18} * \emph{S. Shlosman},  Topological Tverberg Theorem: the proofs and the counterexamples, Russian Math. Surveys, 73:2 (2018), 175–182. arXiv:1804.03120.}

\newcommand{\sisn}{\bibitem[Si69]{Si69} \emph{K. Sieklucki.} Realization of mappings, Fund. Math. 1969. 65. P.~325-343.}

\newcommand{\sios}{\bibitem[Si16]{Si16} \emph{S. Simon,} Average-Value Tverberg Partitions via Finite Fourier Analysis, Israel J. Math., 216 (2016), 891-904, arXiv:1501.04612.}
 

\newcommand{\sknf}{\bibitem[Sk94]{Sk94} \emph{А. Скопенков.} Геометрическое доказательство теоремы
Нойвирта об утолщаемости 2-мерных полиэдров, Math. Notes. 1995. 58:5. P.~1244-1247.}


\newcommand{\skns}{\bibitem[Sk97]{Sk97} \emph{A. Skopenkov,} On the deleted product criterion for embeddability of manifolds in $\R^m$, Comment. Math. Helv. 72 (1997), 543--555.}

\newcommand{\skne}{\bibitem[Sk98]{Sk98} \emph{A. B. Skopenkov.} On the deleted product criterion for embeddability in $\R^m$, Proc. Amer. Math. Soc., 126:8 (1998), 2467-2476.}

\newcommand{\skzz}{\bibitem[Sk00]{Sk00} \emph{A. Skopenkov,} On the generalized Massey--Rolfsen invariant for link maps, Fund. Math. 165 (2000), 1--15.}

\newcommand{\skzt}{\bibitem[Sk02]{Sk02} \emph{A. Skopenkov,} On the Haefliger-Hirsch-Wu invariants for embeddings and immersions, Comment. Math. Helv. 77 (2002), 78--124.}

\newcommand{\skzth}{\bibitem[Sk03]{Sk03} \emph{M. Skopenkov,} Embedding products of graphs into Euclidean spaces,
Fund. Math. 179 (2003),~191--198, arXiv:0808.1199.}

\newcommand{\skzthd}{\bibitem[Sk03']{Sk03'} \emph{M. Skopenkov,} On approximability by embeddings of cycles in the plane, Topol. Appl. 134 (2003),~1--22, arXiv:0808.1187.}

\newcommand{\skzf}{\bibitem[Sk05]{Sk05} * \emph{A. Skopenkov,}
On the Kuratowski graph planarity criterion, Mat. Prosveschenie, 9 (2005), 116-128. arXiv:0802.3820.}


\newcommand{\skzs}{\bibitem[Sk05i]{Sk05i} \emph{A. Skopenkov,} A new invariant and parametric connected sum of embeddings, Fund. Math. 197 (2007) 253--269. arxiv:math/0509621.}

\newcommand{\skzei}{\bibitem[Sk05]{Sk05} \emph{A.  Skopenkov,} A classification of smooth embeddings of
4-manifolds in 7-space, I, Topol. Appl., 157 (2010) 2094--2110. arXiv:math/0512594.}

\newcommand{\skze}{\bibitem[Sk06]{Sk06} * \emph{A. Skopenkov,} Embedding and knotting of manifolds in Euclidean spaces, London Math. Soc. Lect. Notes, 347 (2008) 248--342. arXiv:math/0604045.}

\newcommand{\skzsi}{\bibitem[Sk06']{Sk06'} \emph{A. Skopenkov,} A classification of smooth embeddings of 3-manifolds in 6-space, Math. Zeitschrift, 260:3 (2008) 647--672. arxiv:math/0603429.}

\newcommand{\skzsc}{\bibitem[Sk06c]{Sk06c} \emph{A. Skopenkov,} Classification of embeddings below the metastable dimension, arXiv:math/0607422.}

\newcommand{\skozp}{\bibitem[Sk08]{Sk08} \emph{A.  Skopenkov,} Embeddings of $k$-connected $n$-manifolds into
$\R^{2n-k-1}$. arxiv:math/0812.0263; earlier version published in Proc. Amer. Math. Soc., 138 (2010) 3377--3389.}

\newcommand{\skoz}{\bibitem[Sk10]{Sk10} * \emph{А. Скопенков,} Вложения в плоскость графов с вершинами степени 4,
Мат. просвещение, 21 (2017), arXiv:1008.4940.}

\newcommand{\skoo}{\bibitem[Sk11]{Sk11} \emph{M. Skopenkov,} When is the set of embeddings finite up to isotopy? Intern. J. Math. 26:7 (2015), 28 pp. arXiv:1106.1878.}

\newcommand{\skofo}{\bibitem[Sk14]{Sk14} \emph{A. Skopenkov,} How do autodiffeomorphisms act on embeddings, Proc. A of the Royal Society of Edinburgh, 148:4 (2018), 835--848. arXiv:1402.1853.}

\newcommand{\sks}{\bibitem[Sk14]{Sk14} * \emph{A. Skopenkov,} Realizability of hypergraphs and intrinsic linking  theory, Mat. Prosveschenie, 32 (2024), 125--159, arXiv:1402.0658.}

\newcommand{\sksr}{\bibitem[Sk14]{Sk14} * \emph{А. Скопенков,} Реализуемость гиперграфов и неотъемлемая зацепленность, Мат. просвещение, 32 (2024), 125--159. arXiv:1402.0658.}


\newcommand{\skof}{\bibitem[Sk15]{Sk15} * \emph{А. Скопенков,} Алгебраическая топология с геометрической точки зрения, Москва, МЦНМО, 2015 (1е издание).}

\newcommand{\skofe}{\bibitem[Sk15]{Sk15} * \emph{A. Skopenkov,} Algebraic Topology From Geometric Viewpoint (in Russian), MCCME, Moscow, 2015 (1st edition). }

\newcommand{\skofel}{\bibitem[Sk15e]{Sk15e} * \emph{А. Скопенков,} Алгебраическая топология
с геометрической точки зрения, эл. версия, \url{http://www.mccme.ru/circles/oim/home/combtop13.htm\#photo}}


\newcommand{\skoomp}{\bibitem[Sk11]{Sk11} A. Skopenkov, A simple proof of the Abel-Ruffini theorem (in Russian),
Mat. Prosveschenie, 15 (2011) 113-126, arXiv:1102.2100.}

\newcommand{\skofmp}{\bibitem[Sk15]{Sk15} A. Skopenkov, A short elementary proof of the Ruffini-Abel Theorem (in Russian),
Mat. Prosveschenie, 36 (2026) 95--113. Abridged English version is published in [Sk21m, \S8]; full English version: arXiv:1508.03317.}

\newcommand{\skotzr}{\bibitem[Sk20]{Sk20} * \emph{А. Скопенков,} Алгебраическая топология с геометрической точки зрения, Москва, МЦНМО, 2020 (2е издание).
Обновляемая версия части книги: \url{http://www.mccme.ru/circles/oim/obstruct.pdf}}

\newcommand{\skotz}{\bibitem[Sk20]{Sk20} * \emph{A. Skopenkov,} Algebraic Topology From Geometric Standpoint (in Russian), MCCME, Moscow, 2020 (2nd edition).
Update of a part: \url{http://www.mccme.ru/circles/oim/obstruct.pdf} .
Part of the English translation: \url{https://www.mccme.ru/circles/oim/obstructeng.pdf}.}


\newcommand{\skofp}{\bibitem[Sk15]{Sk15} \emph{A. Skopenkov,} Classification of knotted tori,
Proc. A of the Royal Soc. of Edinburgh, 150:2 (2020), 549-567. Full version: arXiv:1502.04470.}


\newcommand{\skos}{\bibitem[Sk16]{Sk16} * \emph{A. Skopenkov,} A user's guide to the topological Tverberg Conjecture, arXiv:1605.05141v5. Abridged earlier published version: Russian Math. Surveys, 73:2 (2018), 323--353.}


\newcommand{\skosd}{\bibitem[Sk16']{Sk16'} * \emph{A. Skopenkov,} Stability of intersections of graphs in the plane and the van Kampen obstruction, Topol. Appl. 240(2018) 259--269, arXiv:1609.03727.}


\newcommand{\skosc}{\bibitem[Sk16c]{Sk16c} * \emph{A. Skopenkov,}  Embeddings in Euclidean space: an introduction to their classification, to appear in Boll. Man. Atl. 
\url{http://www.map.mpim-bonn.mpg.de/}
(this site is under long maintenance),
\url{https://old.mccme.ru/circles/oim/matlas-emb/Embeddings_in_Euclidean_space:_an_introduction_to_their_classification.html}
 (some Manifold Atlas features are not available on this page).}

\newcommand{\skosie}{\bibitem[Sk16e]{Sk16e} * \emph{A. Skopenkov,} Embeddings just below the stable range: classification, to appear in Boll. Man. Atl.
\url{http://www.map.mpim-bonn.mpg.de/Embeddings_just_below_the_stable_range:_classification}}

\newcommand{\skost}{\bibitem[Sk16t]{Sk16t} * \emph{A. Skopenkov,} 3-manifolds in 6-space, to appear in Boll. Man. Atl. \url{http://www.map.mpim-bonn.mpg.de/3-manifolds_in_6-space}.}

\newcommand{\skosf}{\bibitem[Sk16f]{Sk16f} * \emph{A. Skopenkov,} 4-manifolds in 7-space, to appear in Boll. Man. Atl. \url{http://www.map.mpim-bonn.mpg.de/4-manifolds_in_7-space}.}

\newcommand{\skosh}{\bibitem[Sk16h]{Sk16h} * \emph{A. Skopenkov,} High codimension links, to appear in Boll. Man. Atl.
\linebreak
\url{http://www.map.mpim-bonn.mpg.de/High_codimension_links}.}

\newcommand{\skosi}{\bibitem[Sk16i]{Sk16i} * \emph{A. Skopenkov,} Isotopy, submitted to Boll. Man. Atl.
\url{http://www.map.mpim-bonn.mpg.de/Isotopy}.}

\newcommand{\skosk}{\bibitem[Sk16k]{Sk16k} * \emph{A. Skopenkov,} Knotted tori,
\url{http://www.map.mpim-bonn.mpg.de/Knotted_tori}.}

\newcommand{\skoss}{\bibitem[Sk16s]{Sk16s} * \emph{A. Skopenkov,} Knots, i.e. embeddings of spheres,
\linebreak
\url{http://www.map.mpim-bonn.mpg.de/Knots,_i.e._embeddings_of_spheres}.}

\newcommand{\skose}{\bibitem[Sk17]{Sk17} \emph{A. Skopenkov,}
Eliminating higher-multiplicity intersections in the metastable dimension range. arXiv:1704.00143.}

\newcommand{\skosed}{\bibitem[Sk17v]{Sk17v} * \emph{A. Skopenkov,}
On van Kampen-Flores, Conway-Gordon-Sachs and Radon theorems,  arXiv:1704.00300.}

\newcommand{\sk}{\bibitem[Sk17o]{Sk17o} \emph{A. Skopenkov,} On the metastable Mabillard-Wagner conjecture.  arXiv:1702.04259.}

\newcommand{\skmos}{\bibitem[Sk17d]{Sk17d} \emph{M. Skopenkov}. Discrete field theory: symmetries and conservation laws, arXiv:1709.04788.}

\newcommand{\skoe}{\bibitem[Sk18]{Sk18} * \emph{A. Skopenkov.} Invariants of graph drawings in the plane.
Arnold Math. J., 6 (2020) 21--55; full version: arXiv:1805.10237.}


\newcommand{\skoer}{\bibitem[Sk18]{Sk18} * \emph{А. Скопенков,} Инварианты изображений графов на плоскости,
Мат. просвещение, 31 (2023), 74-127. arXiv:1805.10237.}


\newcommand{\sktthd}{\bibitem[Sk23']{Sk23'} * \emph{A. Skopenkov.} Invariants of graph drawings in the plane (in Russian). Mat. Prosveschenie, 31 (2023), 74-127. arXiv:1805.10237.}

\newcommand{\skoeo}{\bibitem[Sk18o]{Sk18o} * \emph{A. Skopenkov.} A short exposition of S. Parsa's theorems on intrinsic linking and non-realizability. Discr. Comp. Geom. 65:2 (2021), 584--585; full version:  arXiv:1808.08363.}


\newcommand{\skona}{\bibitem[Sk19]{Sk19} * \emph{A. Skopenkov,} A short exposition of the Levine-Lidman example of spineless 4-manifolds, arXiv:1911.07330.}

\newcommand{\sktze}{\bibitem[Sk21m]{Sk21m} * \emph{A. Skopenkov.} Mathematics via Problems. Part 1: Algebra. Amer. Math. Soc., Providence, 2021. Preliminary version: \url{https://www.mccme.ru/circles/oim/algebra_eng.pdf}}

\newcommand{\sktz}{\bibitem[Sk20u]{Sk20u} * \emph{A. Skopenkov.} A user's guide to basic knot and link theory,
in: Topology, Geometry, and Dynamics, Contemporary Mathematics, vol. 772, Amer. Math. Soc., Providence, RI, 2021, pp. 281--309.
Russian version: Mat. Prosveschenie 27 (2021), 128--165. arXiv:2001.01472.}

\newcommand{\sktzru}{\bibitem[Sk20u]{Sk20u} * \emph{А. Скопенков.} Основы теории узлов и зацеплений для пользователя, Мат. просвещение, 27 (2021), 128--165. arXiv:2001.01472.}

\newcommand{\sktzo}{\bibitem[Sk20o]{Sk20o} \emph{A. Skopenkov.} On some results of S. Abramyan and T. Panov, arXiv:2005.11152.}

\newcommand{\sktzr}{\bibitem[Sk20e]{Sk20e} * \emph{A. Skopenkov.}
Extendability of simplicial maps is undecidable, Discr. Comp. Geom., 69:1 (2023), 250--259, arXiv:2008.00492.}


\newcommand{\sktzd}{\bibitem[Sk21d]{Sk21d} * \emph{A. Skopenkov.}
On different reliability standards in current mathematical research, arXiv:2101.03745.
More often updated version: \url{https://www.mccme.ru/circles/oim/rese_inte.pdf}.}

\newcommand{\sktt}{\bibitem[Sk22]{Sk22} * \emph{A. Skopenkov.} Invariants of embeddings of 2-surfaces in 3-space,
arXiv:2201.10944.}

\newcommand{\skttn}{\bibitem[Sk22n]{Sk22n} * \emph{A. Skopenkov}, Netflix problem and realization of (hyper)graphs, \url{https://www.mccme.ru/circles/oim/home/netflix20sep.pdf}}

\newcommand{\sktth}{\bibitem[Sk23]{Sk23} \emph{A. Skopenkov.}  To S. Parsa's theorem on embeddability of joins, arXiv:2302.11537.}

\newcommand{\sktf}{\bibitem[Sk24]{Sk24} * \emph{A. Skopenkov.} Double and triple linking numbers in space (in Russian). Mat. Prosveschenie, 33 (2024), 87--132.}

\newcommand{\sktfr}{\bibitem[Sk24]{Sk24} * \emph{А. Скопенков.} Двойные и тройные коэффициенты зацепления в пространстве. Мат. просвещение, 33 (2024), 87--132.}

\newcommand{\sktfb}{\bibitem[Sk24]{Sk24} \emph{A. Skopenkov.} The band connected sum and the second Kirby move for higher-dimensional links, Stud. Sci. Math. Hung., 62:4 (2025) 320--335, arXiv:2406.15367.}


\newcommand{\sktfe}{\bibitem[Sk24]{Sk24} \emph{A. Skopenkov.}
Embeddings of $k$-complexes in $2k$-manifolds and minimum rank of partial symmetric matrices, arXiv:2112.06636v4.}

\newcommand{\skd}{\bibitem[Sk]{Sk} * \emph{А. Скопенков.} Алгебраическая топология с алгоритмической точки зрения, 
\url{http://www.mccme.ru/circles/oim/algor.pdf}.}

\newcommand{\skde}{\bibitem[Sk]{Sk} * \emph{A. Skopenkov.} Algebraic Topology From Algorithmic Standpoint, draft of a book, mostly in Russian,
\url{http://www.mccme.ru/circles/oim/algor.pdf}.}


\newcommand{\skon}{\bibitem[Skw]{Skw} * \emph{A. Skopenkov.} Whitney trick for eliminating multiple intersections, slides for talks at St Petersburg, Brno, Kiev, Moscow,  \url{https://www.mccme.ru/circles/oim/eliminat_talk.pdf}.}

\newcommand{\skl}{\bibitem[EEF]{EEF} * {\it Proposed by D. Eliseev, A. Enne, M. Fedorov, A. Glebov, N. Khoroshavkina, E. Morozov, A. Skopenkov, R. \v Zivaljevi\'c.}
A user's guide to knot and link theory, \url{https://www.turgor.ru/lktg/2019/3} .}

\newcommand{\skr}{\bibitem[Skr]{Skr} * \emph{A. Skopenkov.} Realizability of hypergraphs, slides for talks,  \url{https://www.mccme.ru/circles/oim/algor1_beamer.pdf}.}


\newcommand{\skt}{\bibitem[Skt]{Skt} * \emph{A. Skopenkov.} Transparent anonymous peer review,
\linebreak
\url{https://www.mccme.ru/circles/oim/home/transp_peer_review.htm} .}

\newcommand{\rslktg}{\bibitem[KRR+]{RRSl} * Towards higher-dimensional combinatorial geometry, presented by
E. Kogan, V. Retinskiy, E. Riabov and A. Skopenkov, \url{https://www.mccme.ru/circles/oim/multicomb.pdf} .}



\newcommand{\sm}{\bibitem[Sm]{Sm} S. Smirnov.}

\newcommand{\sper}{\bibitem[Sp]{Sp} * Sperner's lemma defeats the rental harmony problem, \url{https://www.youtube.com/watch?v=7s-YM-kcKME}.}

\newcommand{\sset}{\bibitem[SS83]{SS83} \emph{Е. В. Щепин, М. А. Штанько.} Спектральный критерий вложимости компактов в евклидовы пространства, Труды Ленинградской Международной Топологической конференции. Л.: Наука, 1983. С.~135-142.}

\newcommand{\ssnt}{\bibitem[SS92]{SS92} \emph{J.~Segal and S.~Spie\.z.} Quasi embeddings and embeddings of polyhedra in $\R^m$,  Topol. Appl., 45 (1992) 275--282.}

\newcommand{\sszt}{\bibitem[SS03]{SS03} \emph{F. W. Simmons and F. E. Su.}
Consensus-halving via theorems of Borsuk-Ulam and Tucker, Math. Social Sciences 45 (2003) 15–25. \url{https://www.math.hmc.edu/~su/papers.dir/tucker.pdf}.}

\newcommand{\ssot}{\bibitem[SS13]{SS13} \emph{M. Schaefer and D. \v Stefankovi\v c.} Block additivity of $\Z_2$-embeddings. In Graph drawing, volume 8242 of Lecture Notes in Comput. Sci., 185--195.
Springer, Cham, 2013. \url{http://ovid.cs.depaul.edu/documents/genus.pdf}}

\newcommand{\sstt}{\bibitem[SS23]{SS23} \emph{A. Skopenkov and O. Styrt,} Embeddability of joinpowers, and minimal rank of partial matrices, arXiv:2305.06339.}

\newcommand{\sssne}{\bibitem[SSS]{SSS} \emph{J. Segal, A. Skopenkov and S. Spie\. z.}
Embeddings of polyhedra in $\R^m$ and the deleted product obstruction, Topol. Appl., 85 (1998), 225-234.}

\newcommand{\sstnf}{\bibitem[SST95]{SST95} \emph{R. S. Simon, S. Spie\. z and H. Toru\'nczyk.}
T\lowercase{HE EXISTENCE OF EQUILIBRIA IN CERTAIN GAMES, SEPARATION FOR FAMILIES OF CONVEX FUNCTIONS
AND A THEOREM OF BORSUK-ULAM TYPE}, Israel J. Math 92 (1995) 1--21.}

\newcommand{\sstzt}{\bibitem[SST02]{SST02} \emph{R. S. Simon, S. Spie\. z and H. Toru\'nczyk.}
E\lowercase{QUILIBRIUM EXISTENCE AND TOPOLOGY IN SOME REPEATED GAMES WITH INCOMPLETE INFORMATION},
Trans. Amer. Math. Soc. 354:12 (2002) 5005-5026.}

\newcommand{\stez}{\bibitem[ST80]{ST80} * {\it H.~Seifert and W.~Threlfall.}
A textbook of topology, v~89 of {\em Pure and Applied Mathematics}.
Academic Press, New York-London, 1980.}


\newcommand{\stzs}{\bibitem[ST07]{ST07} * \emph{А. Скопенков и А. Телишев.}
И вновь о критерии Куратовского планарности графов, Мат. Просвещение, 11 (2007), 159--160.}

\newcommand{\stzse}{\bibitem[ST07]{ST07} * \emph{A. Skopenkov and A. Telishev}, Once again on the Kuratowski graph planarity criterion, Mat. Prosveschenie, 11 (2007), 159-160. arXiv:0802.3820.}

\newcommand{\stos}{\bibitem[ST17]{ST17} \emph{A. Skopenkov  and M. Tancer,}
Hardness of almost embedding simplicial complexes in $\R^d$, Discr. Comp. Geom., 61:2 (2019), 452--463. arXiv:1703.06305.}

\newcommand{\stno}{\bibitem[ST91]{ST91} \emph{S.~Spie\. z and H.~Toru\'nczyk}, Moving compacta in $\R^m$ apart,
Topol. Appl. 41 (1991), 193--204.}

\newcommand{\sttt}{\bibitem[St24]{St24} \emph{M. Starkov,} An example of an `unlinked' set of $2k+3$ points in $2k$-space, arXiv:2402.09002.}

\newcommand{\sunt}{\bibitem[Su]{Su} * \emph{Д. Судзуки.} Основы дзэн-буддизма. Наука дзэн --- ум дзэн. Киев: Преса Украiни. 1992.}

\newcommand{\stwh}{\bibitem[SW]{SW} * \url{http://www.map.mpim-bonn.mpg.de/Stiefel-Whitney_characteristic_classes}}

\newcommand{\sz}{\bibitem[SZ05]{SZ} \emph{T. Sch\"oneborn and G. Ziegler}, The Topological Tverberg Theorem and Winding Numbers, J. Comb. Theory, Ser. A, 112:1 (2005) 82--104, arXiv:math/0409081.}

\newcommand{\szno}{\bibitem[Sz91]{Sz91} \emph{A.~Sz\"ucs,} On the cobordism groups of immersions and embeddings,
Math. Proc. Camb. Phil. Soc., 109 (1991) 343--349.}


\newcommand{\ta}{\bibitem[Ta]{Ta} * Handbook of Graph Drawing and Visualization. ed. by R. Tamassia, CRC Press, 2016.}


\newcommand{\tanfo}{\bibitem[Ta94]{Ta94} \emph{K. Taniyama,} Cobordism, homotopy and homology of graphs in $\R^3$,
Topology 33:3 (1994), 509--523.}

\newcommand{\tanf}{\bibitem[Ta95]{Ta95} \emph{K. Taniyama,} Homology classification of spatial embeddings of a graph, Topol. Appl. 65 (1995) 205--228.}

\newcommand{\tazz}{\bibitem[Ta00]{Ta00} \emph{K. Taniyama,} Higher dimensional links in a simplicial complex embedded in a sphere, Pacific Jour. of Math. 194:2 (2000), 465-467.}

\newcommand{\theo}{\bibitem[Th81]{Th81} * \emph{C.~Thomassen,} Kuratowski's theorem, J.~Graph. Theory 5 (1981), 225--242.}

\newcommand{\tooo}{\bibitem[To11]{To11} \emph{Tonkonog D.} Embedding 3-manifolds with boundary into closed 3-manifolds, Topol. Appl. 158 (2011), 1157-1162. arXiv:1003.3029.}


\newcommand{\tsbzf}{\bibitem[TSB]{TSB} \emph{D. M. Thilikos, M. Serna and H. L. Bodlaender},
Cutwidth I: A linear time fixed parameter algorithm, J. of Algorithms, 56:1 (2005), 1--24.}


\newcommand{\tsbzfd}{\bibitem[TSB05']{TSB05'} \emph{D. M. Thilikos, M. Serna and H. L. Bodlaender},
Cutwidth II: , J. of Algorithms, 56:1 (2005), 25--49.}



\newcommand{\umse}{\bibitem[Um78]{Um78} \emph{B. Ummel.} The product of nonplanar complexes does not imbed in 4-space, Trans. Amer. Math. Soc., 242 (1978) 319--328.}




\newcommand{\vant}{\bibitem[Va92]{Va92} * \emph{V.~A.~Vassiliev.} Complements of discriminants of smooth maps: Topology and applications, Amer. Math. Soc., Providence, RI, 1992 (рус. перевод: В. А. Васильев, Топология дополнений к дискриминантам, Фазис, Москва, 1997).}

\newcommand{\val}{\bibitem[Val]{Val} * \url{https://en.wikipedia.org/wiki/Valknut}}


\newcommand{\vi}{\bibitem[Vi]{Vi} * \emph{O. Viro.}
Some integral calculus based on Euler characteristic, Lect. Notes in Math. 1346.}

\newcommand{\vizt}{\bibitem[Vi02]{Vi02} * \emph{Э. Б. Винберг.} Курс алгебры. Москва. Факториал Пресс. 2002.}

\newcommand{\vizteng}{\bibitem[Vi02]{Vi02} * \emph{E. B. Vinberg.} A Course in Algebra. Graduate Studies in Mathematics, vol. 56. 2003.}

\newcommand{\vinhzs}{\bibitem[VINH07]{VINH07} * \emph{О. Я. Виро, О. А. Иванов, Н. Ю. Нецветаев и В. М. Харламов.}
Элементарная топология, МЦНМО. 2007.}

\newcommand{\vktt}{\bibitem[vK32]{vK32} \emph{E.~R.~van~Kampen}, Komplexe in euklidischen R\"aumen, Abh. Math. Sem. Hamburg, 9 (1933) 72--78; Berichtigung dazu, 152--153.}

\newcommand{\kafo}{\bibitem[vK41]{vK41} \emph{E. R. van Kampen,} Remark on the address of S. S. Cairns,
in Lectures in Topology, 311--313, University of Michigan Press, Ann Arbor, MI, 1941.}

\newcommand{\vo}{\bibitem[Vo96]{vo96} \emph{A. Yu. Volovikov,} On a topological generalization of the Tverberg theorem. Math. Notes 59:3 (1996), 324--326.}

\newcommand{\vopns}{\bibitem[Vo96v]{Vo96v} \emph{A. Yu. Volovikov,} On the van Kampen-Flores Theorem.
Math. Notes 59:5 (1996), 477--481.}


\newcommand{\vznt}{\bibitem[VZ93]{VZ93} \emph{A. Vu\v ci\'c and R. T. \v Zivaljevi\'c}, Note on a conjecture of Sierksma, Discr. Comput. Geom. 9 (1993), 339-349.}

\newcommand{\vzzn}{\bibitem[VZ09]{VZ09} \emph{S. T. Vre\'cica and R. T. \v Zivaljevi\'c},  Chessboard complexes
indomitable, J. of Comb. Theory, Ser. A 118:7 (2011), 2157--2166. arXiv:0911.3512.}


\newcommand{\walst}{\bibitem[Wa62]{Wa62} \emph{C.~T.~C.~Wall}, Classification of $(n-1)$-connected $2n$-manifolds, Ann. of Math., 75 (1962) 163--189.}


\newcommand{\wallss}{\bibitem[Wa67]{Wa67} \emph{C.~T.~C.~Wall.} Classification problems in differential topology, IV, Thickenings, Topology 1966. 5. P. 73--94.}

\newcommand{\waldss}{\bibitem[Wa67m]{Wa67m} \emph{F. Waldhausen.} Eine Klasse von 3-dimensional Mannigfaltigkeiten, I. Invent. Math. 1967. 3. P.~308-333.}

\newcommand{\walsz}{\bibitem[Wa70]{Wa70} \emph{C. T. C. Wall,} Surgery on compact manifolds,
1970, Academic Press, London.}

\newcommand{\wess}{\bibitem[We67]{We67} \emph{C.~Weber.} Plongements de poly\`edres dans le domain metastable, Comment. Math. Helv. 42 (1967), 1--27.}

\newcommand{\whit}{\bibitem[Wl]{Wl} * \url{https://en.wikipedia.org/wiki/Whitehead_link}}

\newcommand{\winum}{\bibitem[Wn]{Wn} * \url{https://en.wikipedia.org/wiki/Winding_number}}

\newcommand{\wrss}{\bibitem[Wr77]{Wr77} \emph{P. Wright.} Covering 2-dimensional polyhedra by 3-manifolds spines.
Topology. 16 (1977), 435--439.}

\newcommand{\wufe}{\bibitem[Wu58]{Wu58} \emph{W. T. Wu.} On the realization of complexes in a euclidean space (in Chinese): I, Sci Sinica, 7 (1958) 251--297; II, Sci Sinica, 7 (1958) 365--387; III, Sci Sinica, 8 (1959) 133--150.}

\newcommand{\wufn}{\bibitem[Wu59]{Wu59} \emph{W.~T.~Wu.} On the isotopy of a finite complex in Euclidean space, I, II, Science Record, N.S. 3:8 (1959) 342--347, 348--351.}

\newcommand{\wusf}{\bibitem[Wu65]{Wu65} * \emph{W. T. Wu.} A Theory of Embedding, Immersion and Isotopy of Polytopes in an Euclidean Space. Peking: Science Press, 1965.}


\newcommand{\yann}{\bibitem[Ya99]{Ya99} \emph{Z. Yang.} Computing Equilibria and Fixed Points: The Solution of Nonlinear Inequalities, Kluwer, Springer Science + Business Media, 1990.}


\newcommand{\zesz}{\bibitem[Ze60]{Ze60} \emph{E. C. Zeeman}, Unknotting spheres in five dimensions, Bull. Amer. Math. Soc. 66 (1960) 198.
\linebreak
\url{https://www.ams.org/journals/bull/1960-66-03/S0002-9904-1960-10431-4/S0002-9904-1960-10431-4.pdf}}

\newcommand{\z}{\bibitem[Ze]{Z} * \emph{E. C. Zeeman}, A Brief History of Topology, UC Berkeley, October 27, 1993, On the occasion of Moe Hirsch's 60th birthday, \url{http://zakuski.utsa.edu/~gokhman/ecz/hirsch60.pdf}.}

\newcommand{\zioz}{\bibitem[Zi10]{Zi10} * \emph{D. \v Zivaljevi\'c}, Borromean and Brunnian Rings,
\url{http://www.rade-zivaljevic.appspot.com/borromean.html}.}

\newcommand{\zioo}{\bibitem[Zi11]{Zi11} * \emph{G. M. Ziegler}, 3N Colored Points in a Plane, Notices of the Amer. Math. Soc., 58:4 (2011), 550-557.}


\newcommand{\zot}{\bibitem[Zi13]{Z13} \emph{A. Zimin.} Alternative proofs of the Conway-Gordon-Sachs Theorems, arXiv:1311.2882.}

\newcommand{\zss}{\bibitem[ZSS]{ZSS} * Элементы математики в задачах: через олимпиады и кружки к профессии.
Сборник под редакцией А. Заславского, А. Скопенкова и М. Скопенкова. М.: МЦНМО, 2018.
Обновляемая версия части книги: \url{http://www.mccme.ru/circles/oim/materials/sturm.pdf}.}

\newcommand{\zsse}{\bibitem[ZSS]{ZSS} Elements of mathematics via problems: from olympiades and math circles to a profession (in Russian), editors A. Zaslavsky, A. Skopenkov, and M. Skopenkov. MCCME, Moscow, 2018,
updated part of the book: \url{http://www.mccme.ru/circles/oim/sturm.pdf}.}


\newcommand{\zu}{\bibitem[Zu]{Zu} \emph{J. Zung.} A non-general-position Parity Lemma,
\url{http://www.turgor.ru/lktg/2013/1/parity.pdf}.}




\aronly{\bito}

\bibitem[BM58]{BM58} \textit{R. Bott, J. Milnor,} On the parallelizability of the spheres, Bull. Amer. Math. Soc.
64 (1958) 87--89.


\dstt

\fkosc
\fkon

\hogr

\info

\kstz

\bibitem[KS21]{KS21} * \emph{E. Kogan and A. Skopenkov.} A short exposition of the Patak-Tancer theorem on non-embeddability of $k$-complexes in $2k$-manifolds,  arXiv:2106.14010v4.

\kstoe
\kutt

\lsne


\dmnse

\nkon


\przse

\aronly{\bibitem[PT3]{PT3} \emph{P. Pat\'ak and M. Tancer.} Embeddings of $k$-complexes into $2k$-manifolds. arXiv:1904.02404v3.}

\pton
 
\skzth
\sks
\skos
\skoeo
\skotz

\aronly{\bibitem[Sk21d]{Sk21d} * \emph{A. Skopenkov.}
On different reliability standards in current mathematical research, arXiv:2101.03745v1.} 

\sktfe


\ssnt
\sstt

\tazz


\end{thebibliography}
\end{document}